\documentclass[11pt]{amsart}
\usepackage{amsfonts}

\usepackage{amssymb, amsmath, amscd, eucal, rotating, enumerate}
\usepackage{todonotes}

\input xy
 \xyoption{all}
\newif\ifPDF
\ifx\pdfoutput\undefined 
        \ifx\pdfversion\undefined
                \PDFfalse
        \fi
\else 
      
        \ifnum\pdfoutput > 0
                \PDFtrue
        \else
                \PDFfalse
        \fi
\fi
\numberwithin{equation}{section}
\theoremstyle{plain}
\newtheorem{proposition}{Proposition}[section]
\newtheorem{theorem}[proposition]{Theorem}
\newtheorem{corollary}[proposition]{Corollary}
\newtheorem{lemma}[proposition]{Lemma}
\newtheorem{claim}{Claim}

\theoremstyle{definition}
\newtheorem{definition}[proposition]{Definition}
\newtheorem{remark}[proposition]{Remark}

\newcommand{\CBbb}{\mathbb C}

\newcommand{\QBbb}{\mathbb Q}

\newcommand{\Acal}{\mathcal A}

\newcommand{\Ecal}{\mathcal E}
\newcommand{\Fcal}{\mathcal F}

\newcommand{\Ical}{\mathcal I}

\newcommand{\Ocal}{\mathcal O}

\newcommand{\Qcal}{\mathcal Q}

\DeclareMathOperator{\rank}{rank}

\DeclareMathOperator{\tr}{tr}
\DeclareMathOperator{\Tr}{Tr}
\DeclareMathOperator{\Gr}{Gr}
\DeclareMathOperator{\gr}{gr}

\DeclareMathOperator{\codim}{codim}
\DeclareMathOperator{\supp}{supp}
\DeclareMathOperator{\sing}{sing}

\newcommand{\dbar}{\bar\partial}
\newcommand{\lra}{\longrightarrow}

\begin{document}
\title[singular sets for the Yang-Mills flow]{Analytic cycles, Bott-Chern
forms, and singular sets for the Yang-Mills flow on K\"ahler manifolds}
\author[Benjamin Sibley]{Benjamin Sibley}
\address{D\'epartment de Math\'ematique,
Universit\'e Libre de Bruxelles,
CP 218,
Boulevard du Triomphe,
B-1050 Bruxelles,
Belgique}
\email{bsibley@ulb.ac.be}
\author[Richard A. Wentworth]{Richard A. Wentworth}
\address{Department of Mathematics, University of Maryland, College Park, MD
20742, USA}
\email{raw@umd.edu}
\thanks{R.W.\ was supported in part by NSF grant DMS-1406513.}
\date{\today}

\begin{abstract}
It is shown that the singular set for the Yang-Mills flow on unstable
holomorphic vector bundles over compact K\"ahler manifolds is completely
determined by the Harder-Narasimhan-Seshadri filtration of the initial
holomorphic bundle. We assign a multiplicity to irreducible top dimensional
components of the singular set of a holomorphic bundle with a filtration by
saturated subsheaves. We derive a singular Bott-Chern formula relating the
second Chern form of a smooth metric on the bundle to the Chern current of
an admissible metric on the associated graded sheaf. This is used to show
that the multiplicities of the top dimensional bubbling locus defined via
the Yang-Mills density agree with  the corresponding multiplicities for the
Harder-Narasimhan-Seshadri filtration. The set theoretic equality of
singular sets is a consequence.
\end{abstract}

\maketitle





\setcounter{tocdepth}{2} 
\thispagestyle{empty}


\baselineskip=16pt



\section{Introduction}

The purpose of this paper is the exact determination of the bubbling locus
for the limit of unstable integrable connections on a hermitian vector
bundle over a compact K\"ahler manifold $(X,
\omega)$ along the Yang-Mills flow.
Roughly speaking, our theorem states that the set of points where curvature
concentration occurs coincides with a  subvariety canonically determined by
a certain filtration of the initial holomorphic bundle by saturated
subsheaves. This result builds on work of several authors on both the
analytic and algebraic sides of this picture, and so below we present a
brief description of some of this background.

From the analytic point of view, the original compactness result that
informs all subsequent discussion is that of Uhlenbeck \cite{Uhlenbeck82a}
which, combined with the result of the unpublished preprint \cite{Uhlenbeck}
implies that a sequence $A_{i}$ of integrable unitary connections on a
hermitian vector bundle $E\to X$ with uniformly bounded Hermitian-Einstein
tensors $\sqrt{-1}\Lambda_\omega F_{A_i}$ has a subsequence that weakly converges
locally, modulo gauge transformations and in a certain Sobolev norm outside
a set $Z^{an}$ of Hausdorff (real) codimension at least $4$, to a unitary
connection $A_{\infty }$ (see Theorem \ref{thm:uhlenbeck} below). This
played an important role in the fundamental work of Uhlenbeck and Yau \cite%
{UhlenbeckYau86}. We call $Z^{an}$ the \emph{analytic singular set}\emph{\ }%
(or \emph{bubbling set}). If $A_\infty$ is Yang-Mills, then by the removable
singularities theorem it extends to a unitary connection on a hermitian
vector bundle $E_{\infty }$ defined off a set of (real) codimension at least 
$6$. We call this extension an \emph{Uhlenbeck limit} and note that $E_\infty
$ may be topologically distinct from $E$ on its set of definition. The
corresponding holomorphic bundle also extends as a reflexive coherent
analytic sheaf $\mathcal{E}_\infty\to X$.

The long time existence of the Yang-Mills flow on integrable connections
over K\"ahler manifolds was proved by Donaldson \cite{Donaldson85}.  Hong
and Tian have shown \cite{HongTian04}, using a combination of blow-up analysis of the sequence
near the singular set and geometric measure theory techniques, that in fact
the convergence can be taken to be $C^{\infty }$ away from the bubbling set
and that the bubbling set itself is a holomorphic subvariety. This second
fact relies on a structure theorem of King \cite{King71} or a generalized
version \cite{HS} due to Harvey and Shiffman, for closed, positive,
rectifiable currents on complex manifolds. More precisely, in \cite{Tian00}
Tian gives a decomposition $Z^{an}=Z_{b}^{an}\cup \sing  A_{\infty }$ into a
rectifiable piece $Z_{b}^{an}$ and a set $\sing  A_{\infty }$ over which the
connection $A_{\infty }$ cannot be extended, and having zero $(2n-4)$%
-dimensional Hausdorff measure, where $n$ is the complex dimension of $X $.
The results in \cite{Tian00} together with the main result of \cite{King71}
or \cite{HS} imply that $Z_{b}^{an}$ is a subvariety of pure complex
codimension $2$. Then a result contained in \cite{TianYang02} shows that $%
\sing A_{\infty }$ is also an analytic subvariety of codimension at least $3$%
. Moreover, weak convergence of the measures defined by the Yang-Mills
densities $\left\vert F_{A_{i}}\right\vert ^{2}dvol_{\omega }$ leads to the
definition of a natural density function $\Theta $ supported on $Z^{an}$. In 
\cite{Tian00} Tian shows that $\Theta$ assigns an integer weight to each
irreducible codimension $2$ component $Z$ of $Z_{b}^{an}$. We call this
weight the \emph{analytic multiplicity} $m_Z^{an}$ of the component $Z$.

On the algebraic side, associated to a holomorphic vector bundle $\mathcal{E}%
\rightarrow X$ we have the Harder-Narasimhan-Seshadri (HNS) filtration and
its associated graded sheaf $\Gr (\mathcal{E})$, which is locally free away
from a complex analytic subvariety $Z^{alg}$ of codimension $\geq 2$. The
sheaf $\Gr (\mathcal{E})$ is uniquely determined up to isomorphism by $%
\mathcal{E}$ and the K\"ahler class $[\omega]$, and therefore so is $Z^{alg}$%
, which we call the \emph{algebraic singular set} (the terminology, which
has taken hold, is a bit inaccurate in the sense that we do \emph{not}
assume that $X$ be a projective algebraic manifold). The reflexive sheaf $%
\Gr (\mathcal{E})^{\ast \ast }$ is locally free outside a subvariety of
codimension $\geq 3$. 
The restriction of the torsion sheaf $\Gr 
(\mathcal{E})^{\ast \ast }/\Gr (\mathcal{E})$ has a generic rank on each
irreducible codimension $2$ component $Z$ of $Z^{alg}$, and we call this
rank the \emph{algebraic multiplicity} $m_Z^{alg}$ of the component $Z$.

A hermitian metric on the locally free part of a torsion-free sheaf on $X$
is called \emph{admissible} if its Chern connection has square integrable
curvature and bounded Hermitian-Einstein tensor (see Section \ref%
{sec:stability}). By the main result of Bando-Siu \cite{BandoSiu94}, the
sheaf $\Gr (\mathcal{E})^{\ast \ast }$ carries an admissible
Hermitian-Einstein metric whose Chern connection is the direct sum of the
Hermitian-Yang-Mills connections on its stable summands, and this is unique
up to gauge. 
The main result linking the two pictures described above is the fact that in
the case of Uhlenbeck limits along the Yang-Mills flow, $\mathcal{E }%
_{\infty }$ is actually isomorphic to $\Gr (\mathcal{E})^{\ast \ast }$. In
particular, the limiting connection $A_{\infty }$ is gauge equivalent to the
Bando-Siu connection and is independent of the choice of subsequence. This
result is due to Daskalopoulos \cite{Daskal92} and R\aa de \cite{Rade92} for 
$\dim _{\mathbb{C} }X=1$ (where there are no singularities),
Daskalopoulos-Wentworth \cite{DaskalWentworth04} for $\dim _{\mathbb{C} }X=2$%
, and Jacob \cite{J1, J2, J3} and Sibley \cite{Sibley} in higher dimensions.
A priori, $Z^{an}$ also depends on a choice of subsequence along the
Yang-Mills flow, whereas $Z^{alg}$ is uniquely determined as previously
mentioned. In $\dim_\mathbb{C }X\geq 3$, the identification of the limiting
structure $A_\infty$ just mentioned does not address the relationship
between  these two singular sets; to do so is the aim of this paper. We now
formulate our main theorem as follows.

\begin{theorem}
\label{thm:main} Let $\mathcal{E}\to X$ be a hermitian holomorphic vector
bundle over a compact K\"ahler manifold with Chern connection $A_0$, and let
$Z^{alg}$ denote the algebraic singular set associated to the
Harder-Narasimhan-Seshadri filtration of $\mathcal{E}$. Let $A_t$, $0\leq
t<+\infty$, denote the Yang-Mills flow with initial condition $A_0$. Then:

\begin{enumerate}[(1)] 
\item For any sequence $t_{i}\to \infty$ defining an Uhlenbeck limit $%
A_{t_i}\to A_\infty$ with bubbling set $Z^{an}$, then $Z^{an}=Z^{alg}$ as
sets.

\item Modulo unitary gauge transformations, $A_{t}\to A_\infty$ smoothly
away from $Z^{alg}$ as $t\to\infty$ to the admissible Yang-Mills connection $%
A_\infty$ on a reflexive sheaf isomorphic to $\Gr(\mathcal{E})^{\ast\ast}$.

\item For any irreducible component $Z\subset Z^{alg}$ of complex
codimension $2$, then $Z\subset Z_b^{an}$ and the analytic and algebraic
multiplicities of $Z$ are equal.
\end{enumerate}
\end{theorem}

\begin{remark}
\label{rem:intro}

\begin{enumerate}
\item Theorem \ref{thm:main} generalizes to higher dimensions the result of 
\cite{DaskalWentworth07} in the case $\dim _{\mathbb{C}}X=2$.

\item Item (2) follows from (1) by combining the work of Hong-Tian \cite{HongTian04} (for the
smooth convergence) and Jacob \cite{J1,J2,J3}, Sibley \cite{Sibley} (for the identification of the limit).

\item Under certain technical assumptions on the growth of norms of the
second fundamental forms associated to the HNS filtration near $Z^{an}$,
Collins and Jacob \cite{CollinsJacob12} show that (1) holds.  The proof of
(1) in this paper relies on the structure theorems of Tian \cite{Tian00} and
King/Harvey-Shiffman \cite{HS,King71}, and the equality of multiplicities from item (3).
\end{enumerate}
\end{remark}

A key step in the proof of Theorem \ref{thm:main} is a singular version of
the usual Bott-Chern formula which is of independent interest. Suppose $%
\mathcal{E}\to X$ is a holomorphic bundle with a filtration by saturated
subsheaves and associated graded sheaf $\Gr(\mathcal{E})$. Given hermitian
structures, then since $\mathcal{E}$ and $\Gr(\mathcal{E})$  are
topologically isomorphic away from the singular set $Z^{alg}$ the Bott-Chern
formula relates representatives of the second Chern characters $ch_2$ in
terms of Chern connections  as an equation of smooth forms outside this set.
If the hermitian metric on $\Gr(\mathcal{E})$ is admissible, we can extend
this equality over the singular set as an equation of currents, at the cost
of introducing on one side of the equation the current defined by the
analytic cycle associated to the irreducible codimension $2$ components $%
\{Z_j^{alg}\}$ of $Z^{alg}$ with the multiplicities $m_j^{alg}$ defined
above. The result is the following

\begin{theorem}
\label{thm:bottchern} Let $\mathcal{E}\to X$ be a holomorphic vector bundle
with a filtration by saturated subsheaves and hermitian metric $h_0$. If $h$
is an admissible metric on $\Gr(\mathcal{E})$, then the following equation
of closed currents holds:%
\begin{equation}  \label{eqn:bc-current}
ch_{2}(\Gr(\mathcal{E}), h)-ch_{2}(\mathcal{E},h_0)
=\sum\limits_{j}m_{j}^{alg}Z_{j}^{alg}+dd^{c}\Psi
\end{equation}
where $\Psi=\Psi(h, h_0)$ is a $(1,1)$-current on $X$, smooth away from $%
Z^{alg}$. Here, $Z_{j}^{alg}$ is regarded as a $(2,2)$-current by
integration over its set of smooth points, and $ch_{2}(\Gr(\mathcal{E}),h)$
denotes the extension of the Chern form \eqref{eqn:ch2} as a current on $X$. 
\end{theorem}

The organization of this paper is as follows. In Section $2$ we recall the
definition of the Yang-Mills functional and its negative gradient flow,
along with the statement of the main result of \cite{BandoSiu94}. We review
the HNS filtration and give a precise definition of the associated algebraic
multiplicity. We also recall the version of Uhlenbeck compactness that
applies to integrable connections with bounded Hermitian-Einstein tensors.
We describe the work of \cite{Tian00} and \cite{HongTian04} in a bit more
detail and elaborate the notion of analytic multiplicity.

Section $3$ is devoted to the proof of Theorem \ref{thm:bottchern}. It will
be shown that \eqref{eqn:bc-current} is essentially a consequence of the
cohomological statement that the second Chern character of the torsion sheaf 
$\Gr (\mathcal{E})^{\ast \ast }/ \Gr (\mathcal{E})$ is equal to the analytic
cycle appearing on the right hand side of \eqref{eqn:bc-current}. This
latter fact is probably well-known to algebraic geometers, and in the
projective case it can be obtained from the Grothendieck-Riemann-Roch
theorem of Baum-Fulton-MacPherson \cite{BFM1, BFM2}. In the setting of
arbitrary compact complex manifolds the desired identity, Proposition \ref%
{prop:key}, follows from the generalization of BFM due to Levy \cite{Levy87}
. The statement in that reference is written in terms of analytic and
topological $K$-theory,  and much of Section \ref{sec:grr} is therefore
devoted to recalling this formalism and using it to obtain the statement
about the second Chern character mentioned above. In order to go further, we
also need to prove that the second Chern current $ch_2(\Gr(\mathcal{E}),h)$
for an admissible metric $h$ is closed and represents $ch_2(\Gr(\mathcal{E}%
)^{\ast\ast})$ in cohomology (Proposition \ref{prop:ch2}). The proof relies
on the monotonicity formula and $L^p$-estimates derived by Uhlenbeck in the
aforementioned paper \cite{Uhlenbeck}, as well as an argument of Tian \cite%
{Tian00} which was used in the case of admissible Yang-Mills connections. We
should point out that other versions of singular Bott-Chern currents exist
in the literature (e.g.\ the work of Bismut-Gillet-Soul\'e \cite%
{BismutGilletSoule90}). 

In Section $4$ we prove a slicing lemma showing that the analytic
multiplicity may be computed by restricting to a (real) $4$-dimensional
slice through a generic smooth point of an irreducible component of the
analytic singular set (Lemma \ref{lem:slice}). Since a parallel result holds
for the currents of integration appearing in the Bott-Chern formula,  we can
use this and Theorem \ref{thm:bottchern} to compare algebraic and analytic
multiplicities. Combined with an argument similar
to that used in \cite{DaskalWentworth07}, this leads to a proof of the main
theorem.

\section{Preliminaries}

\subsection{Stability, Hermitian-Einstein metrics, and the Yang-Mills flow}

\label{sec:stability}

Unless otherwise stated, $X$ will be a compact K\"{a}hler manifold of
complex dimension $n$ with K\"{a}hler form $\omega $. Let $\Lambda_{\omega }
$ denote the formal adjoint of the Lefschetz operator given by wedging with $%
\omega $. Let $\mathcal{E}\to X$ be a hermitian holomorphic vector bundle
with metric $h$, Chern connection $A=(\mathcal{E}, h)$, and curvature $F_A$.
Then $\sqrt{-1}\Lambda_{\omega }F_{A}$ is a hermitian endomorphism of the
underlying hermitian bundle $(E, h)$, and it is called the \emph{%
Hermitian-Einstein tensor}. The equality 
\begin{equation}  \label{eqn:he}
\sqrt{-1}\Lambda_{\omega }F_{A}=\mu\, \mathbf{I}_E
\end{equation}
where $\mu$ is constant may be viewed as an equation for the metric $h$. A
solution to \eqref{eqn:he} is called a \emph{Hermitian-Einstein metric}, and
the corresponding Chern connection is called \emph{Hermitian-Yang-Mills}.

If $\mathcal{E }\rightarrow X$ is a torsion-free sheaf, then its $\omega$%
-slope is defined 
\begin{equation*}
\mu _{\omega }(\mathcal{E })=\frac{1}{\rank \mathcal{E }}\int_{X}c_{1}( 
\mathcal{E })\wedge \omega ^{n-1}.
\end{equation*}
Then $\mathcal{E }$ is called \emph{stable} (resp.\ \emph{semistable}), if $%
\mu _{\omega }(\mathcal{F})< \mu _{\omega }(\mathcal{E })$ (resp.\ $\leq$)
for all coherent  subsheaves $\mathcal{F}\subset\mathcal{E}$ with $0<\rank%
\mathcal{F}<\rank\mathcal{E}$. The term \emph{polystable} refers to a sheaf
which splits holomorphically into a direct sum of stable sheaves, all of the
same slope. The Donaldson-Uhlenbeck-Yau theorem \cite%
{Donaldson85,Donaldson87,UhlenbeckYau86} states that a holomorphic bundle $%
\mathcal{E}$ admits a Hermitian-Einstein metric if and only if the bundle is
polystable. If the volume of $X$ is normalized to be $2\pi /(n-1)!$, then
the constant in \eqref{eqn:he} is $\mu=\mu _{\omega }(E)$.

A hermitian metric $h$ on the locally free part of a torsion-free sheaf $%
\mathcal{E}$ is called $\omega$-\emph{admissible} if $\left| \Lambda_{\omega }F_{A}\right|_h\in L^{\infty }(X)$ and $\left| F_{A}\right|_h\in
L^{2}(X,\omega )$, where $A$ is the Chern connection $(\mathcal{E}, h)$. The
Hermitian-Einstein condition \eqref{eqn:he} has the same meaning in this
context. We will sometimes refer to $A$ as an \emph{admissible connection}.
The notion of admissibility was introduced by Bando and Siu who also proved
the  Hitchin-Kobayashi correspondence in this context.

\begin{theorem}[Bando-Siu \protect\cite{BandoSiu94}]
\label{thm:bando-siu} Let $\mathcal{E}\to X$ be a torsion-free coherent
sheaf with reflexivization $\mathcal{E}^{\ast\ast}$. Then:

\begin{enumerate}
\item there exists an admissible metric on $\mathcal{E}$;

\item any admissible metric on $\mathcal{E}$ extends to a metric on the
locally free part of $\mathcal{E}^{\ast\ast}$ which is in $L^p_{2,loc.}$ for
all $p$;

\item there exists an admissible Hermitian-Einstein metric on $\mathcal{E}%
^{\ast\ast}$ if and only $\mathcal{E}^{\ast\ast}$ is polystable.
\end{enumerate}
\end{theorem}

The \emph{Yang-Mills flow} of unitary connections on a hermitian bundle $%
(E,h)$ is given by the equations: 
\begin{equation}  \label{eqn:ymflow}
\frac{\partial A_{t}}{\partial t}=-d_{A_{t}}^{\ast }F_{A_{t}}\ ,\ A(0)=A_{0}
\end{equation}
Donaldson \cite{Donaldson85} shows that if $X$ is K\"ahler and $A_0$ is
integrable, then a solution to \eqref{eqn:ymflow} exists (modulo gauge
transformations) for all $0\leq t<+\infty$. Eq.\ \eqref{eqn:ymflow} is
formally the negative gradient flow for the \emph{Yang-Mills functional}:%
\begin{equation*}
YM(A)=\int_{X}\left\vert F_{A}\right\vert ^{2}dvol_{\omega }.
\end{equation*}%
Critical points of this functional are called \emph{Yang-Mills connections}
and satisfy $d_{A}^{\ast }F_{A}=0$. By the K\"ahler identities, if $E$
admits an integrable Yang-Mills connection $A$, then it decomposes
holomorphically and isometrically  into a direct sum of the (constant rank)
eigenbundles of $\sqrt{-1}\Lambda_{\omega }F_{A}$, and the induced
connections are Hermitian-Yang-Mills. Similarly, an admissible Yang-Mills
connection on a reflexive sheaf gives a direct sum decomposition into
reflexive sheaves admitting admissible Hermitian-Yang-Mills connections.

\subsection{The HNS filtration and the algebraic singular set}

\label{sec:algsingset}


Let $\mathcal{E}\rightarrow X$ be a holomorphic bundle. We say that $%
\mathcal{E}$ is \emph{filtered by saturated subsheaves} if there are
coherent subsheaves 
\begin{equation}
0=\mathcal{E}_{0}\subset \mathcal{E}_{1}\subset \cdots \subset \mathcal{E}%
_{\ell }=\mathcal{E}  \label{eqn:filtration}
\end{equation}%
with torsion-free quotients $\mathcal{Q}_{i}=\mathcal{E}_{i}/\mathcal{E}%
_{i-1}$. The associated graded sheaf is 
\begin{equation*}
\Gr(\mathcal{E})=\bigoplus_{i=1}^{\ell }\mathcal{Q}_{i}
\end{equation*}%
\noindent We define the \emph{algebraic singular set} of a filtration to be $%
Z^{alg}=\sing\Gr(\mathcal{E})$, i.e.\ the complement of the open set where $%
\Gr(\mathcal{E})$ is locally free. Since the $\mathcal{Q}_{i}$ are
torsion-free, the singular set is an analytic subvariety of codimension at
least $2$. Note that we have%
\begin{equation}
Z^{alg}=\supp\left( \Gr(\mathcal{E})^{\ast \ast }/\Gr(\mathcal{E}%
)\right) \cup \sing\Gr(\mathcal{E})^{\ast \ast }  \label{eqn:algsingset}
\end{equation}%
since if $x\notin \supp(\Gr(\mathcal{E})^{\ast \ast }/\Gr(\mathcal{E%
}))$, and  $x\notin \sing\Gr(\mathcal{E})^{\ast \ast }$, then $\Gr(\mathcal{E%
})$ must be locally free at $x$. 
We also record the simple

\begin{lemma} \label{lem:sheaf}
Each $\mathcal{E }_{i}$ is reflexive. Moreover, $\Ecal_i$ is locally free on $X- Z^{alg}$.
\end{lemma}

\begin{proof}
Consider the last  quotient
\begin{equation*}
0\longrightarrow \mathcal{E }_{\ell-1}\longrightarrow \mathcal{E }
\longrightarrow \mathcal{Q}_{\ell}\longrightarrow 0,
\end{equation*}%
Since $\Ecal$ is reflexive and $\Qcal_\ell$ torsion-free, $\Ecal_{\ell-1}$ is reflexive (cf.\ \cite[Prop.\ V.5.22]{Kobayashi87}).  The result for $\Ecal_i$ follows by repeatedly applying this argument.
For the second statement, start with 
the first quotient $\mathcal{Q}_{1}=\mathcal{E }_{1}$.
Then there is an exact sequence:%
\begin{equation*}
0\longrightarrow \mathcal{E }_{1}\longrightarrow \mathcal{E }%
_{2}\longrightarrow \mathcal{Q}_{2}\longrightarrow 0,
\end{equation*}%
If both $\mathcal{E}_1$ and $\mathcal{Q}_{2}$ are locally free, then this
sequence splits at the level of stalks, and so $\mathcal{E }_{2}$ is also
locally free. Iterating this argument proves the claim.
\end{proof}

The following a priori structure of $Z^{alg}$ will also be important.

\begin{proposition} \label{prop:scheja}
On the complement of $\sing \Gr(\Ecal)^{\ast\ast}$, $Z^{alg}$ has pure codimension $2$.
\end{proposition}

\begin{proof}
  Choose a coordinate ball $B_\sigma(x)\subset X-\sing \Gr(\Ecal)^{\ast\ast}$,  and assume that ${\rm codim}(B_\sigma(x)\cap Z^{alg})\geq 3$. Set $U=B_\sigma(x)-Z^{alg}$.
The first step in the filtration is:
\begin{equation} \label{eqn:exact1}
0\lra \Ecal_1\lra \Ecal_2\stackrel{ q_2 }{\lra} \Qcal_2\lra 0
\end{equation}
 By Lemma \ref{lem:sheaf}, 
 $\Ecal_1$ and $\Ecal_2$ are reflexive. Furthermore, since $B_\sigma(x)$ misses $\sing \Gr(\Ecal)^{\ast\ast}$,
 $\Ecal_1=\Ecal_1^{\ast\ast}$ and $\Qcal_2^{\ast\ast}$ are locally free.  Since $\Qcal_2^\ast=(\Qcal_2^{\ast\ast})^\ast$ is also locally free, tensoring \eqref{eqn:exact1} by $\Qcal_2^\ast$
 leaves the  sequence exact. Consider the resulting exact sequence in cohomology:
\begin{equation} \label{eqn:exact2}
H^0(U, \Ecal_2\otimes\Qcal_2^\ast)\lra H^0(U,\Qcal_2\otimes\Qcal_2^\ast)
\lra H^1(U, \Ecal_1\otimes\Qcal_2^\ast)
\end{equation}
Since $\Ecal_1\otimes\Qcal_2^\ast$ is locally free on $B_\sigma(x)$ and ${\rm codim}(B_\sigma(x)\cap Z^{alg})\geq 3$, it follows
from Scheja's theorem \cite[Sec.\ 3, Satz 3]{Scheja61} that $H^1(U, \Ecal_1\otimes\Qcal_2^\ast)\simeq H^1(B_\sigma(x), \Ecal_1\otimes\Qcal_2^\ast)=\{0\}$.  In particular, 
 the image of the identity ${\bf I}_{\Qcal_2}\in H^0(U,{\rm Hom}(\Qcal_2,\Qcal_2))=H^0(U,\Qcal_2\otimes\Qcal_2^\ast)$ by the coboundary map, is trivial. By exactness of \eqref{eqn:exact2} this means there is $\varphi: \Qcal_2\to \Ecal_2$ on $U$ satisfying $q_2 \circ\varphi={\bf I}_{\Qcal_2}$.  By normality, $\varphi$ extends to a map $\tilde \varphi:\Qcal_2^{\ast\ast}\to \Ecal_2$ on $B_\sigma(x)$. If  $\tilde q_2  : \Ecal_2\to \Qcal_2^{\ast\ast}$ is the  map obtained by composing $q_2 $ with the inclusion $\Qcal_2\hookrightarrow\Qcal_2^{\ast\ast}$, then 
clearly $\tilde q_2\circ \tilde\varphi= {\bf I}_{\Qcal_2^{\ast\ast}}$.  In particular, $\tilde q_2 $ is surjective, and hence $\Qcal_2=\Qcal_2^{\ast\ast}$. So there is no contribution to $Z^{alg}$ from this term in the filtration. Moreover,
since $\Ecal_1$ and $\Qcal_2$ are locally free, eq.\ \eqref{eqn:exact1} implies as in the previous lemma that $\Ecal_2$ is locally free on $B_\sigma(x)$. Now consider the next step in the filtration:
$
0\to \Ecal_2\to \Ecal_3\stackrel{q_3 }{\lra} \Qcal_3\to 0
$.
Again, $\Ecal_2$ and $\Qcal_3^{\ast\ast}$ are locally free on $B_\sigma(x)$, and the argument proceeds as above.  Continuing in this way, we conclude that $ B_\sigma(x)\cap Z^{alg}=\emptyset$. The statement in the proposition follows.
\end{proof}


The main example of interest in this paper is the following. Recall that for
any reflexive sheaf $\mathcal{E }$ on a K\"{a}hler manifold $X$ there is a
canonical filtration of $\mathcal{E }$ by saturated subsheaves $\mathcal{E }%
_{i}$ whose successive quotients are torsion-free, semistable. The slopes $%
\mu_i=\mu(\mathcal{Q}_i)$ satisfy $\mu_1>\mu_2>\cdots>\mu_\ell$. This
filtration is called the \emph{Harder-Narasimhan filtration} of $\mathcal{E}$%
. Moreover, there is a further filtration of the quotients by subsheaves so
that the successive quotients are stable. We call this the \emph{%
Harder-Narasimhan-Seshadri (HNS) filtration}. The associated graded sheaf $%
\Gr(\mathcal{E})$ is a direct sum of stable torsion-free sheaves. It depends
on the choice of K\"ahler class $[\omega]$ but is otherwise canonically
associated to $\mathcal{E}$ up to permutation of isomorphic factors. The $(%
\rank\mathcal{E})$-vector $\vec\mu$ obtained by repeating each $%
\mu_1,\ldots, \mu_\ell$, $\rank\mathcal{Q}_i$ times, is called the \emph{%
Harder-Narasimhan type} of $\mathcal{E}$.

Two more remarks:

\begin{itemize}
\item Strictly speaking, the HNS construction gives rise to a \emph{double}
filtration; however, this fact presents no difficulties, and for simplicity
we shall treat the HNS filtration like the general case.  For more details
the reader may consult \cite{DaskalWentworth04}, \cite{Sibley}, or the book 
\cite{Kobayashi87}.

\item We sometimes use the same notation $\Gr(\mathcal{E})$ for the
associated graded of a general filtration as well as for the HNS filtration
of $\mathcal{E}$. The context will hopefully make clear which is meant.
\end{itemize}

%
%
%
%

\subsection{Multiplicities associated to the support of a sheaf}

If $\mathcal{F}$\ \ is a coherent sheaf on a complex manifold $X$, then the
support  $Z=\supp \mathcal{F}$ is a closed, complex analytic subvariety of $X
$. Moreover, $\supp \mathcal{F}$ is the vanishing set $V(\mathcal{I}_{%
\mathcal{F}})$, where $\mathcal{I}_{\mathcal{F}} \subset \mathcal{O}_{X}$ is
the annihilator ideal sheaf whose presheaf on an open set $U$ is the subset
of functions $\mathcal{O}_{X}(U)$ that annihilate all local sections in $%
\mathcal{F}(U)$. This ideal gives $\supp\mathcal{F}$ the structure of a
complex analytic subspace with structure sheaf $\mathcal{O}_{Z}=\mathcal{O}%
_{X}/\mathcal{I}_{\mathcal{F}}$. Now $Z$ has a decomposition $Z=\cup_{j}Z_{j}
$ into irreducible components $Z_{j}$, with structure sheaves $\mathcal{O}%
_{Z_{j}}$. If necessary, we take the reduced structures, so that each $Z_{j}$
is a reduced and irreducible complex subspace of $X$ with ideal $\mathcal{I}%
_{j}$ (i.e.\ $\supp \mathcal{O}_{X}/\mathcal{I}_{j}=Z_{j}$). The fact that $%
Z_{j}$ is irreducible means that the complex manifold $Z_{j}-\sing Z_{j}$ is
connected.

Note that for the inclusion $\imath :Z_{j}\hookrightarrow X$, the
restriction $\imath ^{\ast }\mathcal{F}$ of $\mathcal{F}$ to $Z_{j}$ is a
coherent sheaf of $\mathcal{O}_{Z_{j}}$ modules. In this way we may regard $%
\mathcal{F}$ as a sheaf on $Z_{j}$. The fibres of $\mathcal{F}$ on $Z$ are
the finite dimensional $\CBbb$-vector spaces 
\begin{equation*}
\mathcal{F}(z)=\mathcal{F}_{z}/\mathfrak{m}_{z}\mathcal{F}_{z}=\mathcal{F}%
_{z}\otimes_{\mathcal{O}_{Z_j,z}} \CBbb,
\end{equation*}%
where $\mathfrak{m}_{z}$ is the maximal ideal in the local ring $\mathcal{O}%
_{Z_{j},z}$, and the rank, $\rank _{z}\mathcal{F}$ at a point is the
dimension of this vector space.

It is not difficult to see (cf.\ \cite[p.\ 91]{GR}) that $\imath^\ast%
\mathcal{F}$ is locally free at a point $z_{0}\in Z_{j}$ if and only if the
function $z\mapsto \rank _{z}\mathcal{F}$ is constant for $z\in Z_j$ near $%
z_{0}$. Therefore away from the set $\sing Z_{j}\cup \sing \imath^\ast 
\mathcal{F}$ of points where $\imath^\ast\mathcal{F}$ fails to be locally
free and $Z_{j}$ is singular, this function is constant (since the set of
smooth points of $Z_{j}$ is connected). The set $\sing Z_{j}\cup \sing %
\imath^\ast \mathcal{F}$ is a proper subvariety of $Z_{j}$, and nowhere
dense in $Z_{j}$ (since $Z_{j}$ is reduced). In particular, this subvariety
has dimension less than $Z_{j}$, and therefore the generic rank of $%
\imath^\ast\mathcal{F}$ on $Z_{j} $ is well-defined. Another way of saying
this is that if $Z_{j}$ has codimension $k$, $z\in Z_{j}$ is a generic
smooth point, and $\Sigma $ is a (locally defined) complex submanifold of $X$
of dimension $k$ intersecting $Z_{j}$ transversely at $z$, then the $\CBbb$-vector space $(\mathcal{O}_{\Sigma })_{z}/(\mathcal{I}_{j}\bigr|%
_\Sigma )_{z} $ is finite dimensional, and its dimension is generically
constant and equal to the rank of $\imath^\ast\mathcal{F}$.

\begin{definition}
\label{def:algmult} \bigskip Given a coherent sheaf $\mathcal{F}$ on $X$ and
an  irreducible component $\imath: Z\hookrightarrow\supp \mathcal{F}$,
define the \emph{multiplicity} $m_{Z}$ of $\mathcal{F}$ along $Z$ to be the
rank of $\imath^\ast\mathcal{F }$.
\end{definition}

For any $k$ we can define an $(n-k)$-cycle associated to $\mathcal{F}$ by:%
\begin{equation}  \label{eqn:cycle}
\left[ \mathcal{F}\right] _{k}=\mathop{\sum_{{\rm irred. }\, Z\subset\supp
\mathcal{F}}}_{\codim Z=k}m_{Z}[Z]
\end{equation}
Of particular interest in this paper is the associated graded sheaf $%
\mathcal{F}=\mathrm{Gr }(\mathcal{E})$ of a locally free sheaf $\mathcal{E}$
with a filtration by saturated subsheaves.  The quotient by the inclusion $%
\Gr (\mathcal{E})\hookrightarrow \Gr (\mathcal{E})^{\ast \ast }$ yields a
torsion sheaf which has support in codimension $2$. The irreducible
components $\{Z_{j}^{alg}\}$ of $Z^{alg}$ with  codimension $=2$ have
associated algebraic multiplicities from Definition \ref{def:algmult}. We
will denote these by $m_{j}^{alg}$ and refer to them as the \emph{algebraic
multiplicities} of the filtration.

Note that in the case $\dim _{\CBbb}X=2$, $\Gr (\mathcal{E}%
)^{\ast \ast }$ is locally free and $\Gr (\mathcal{E})^{\ast \ast }/\Gr (%
\mathcal{E})$ is supported at points. The structure sheaf of singular point $%
z$ is $\mathcal{O}_{X,z}/\mathfrak{m}_{z}=\CBbb$, so the fibre
at this point is just the stalk, and the multiplicity is the $\CBbb$-dimension of the stalk. This was the definition of $m_z^{alg}$ used in 
\cite{DaskalWentworth07}.

\subsection{Uhlenbeck limits and the analytic singular set}

\label{sec:analytic}

We briefly recall the Uhlenbeck compactness theorem. It is a combination of
the results of \cite{Uhlenbeck82a}, \cite{Uhlenbeck}, and \cite{Nakajima88} (see also Theorem $%
5.2$ of \cite{UhlenbeckYau86}).

\begin{theorem}[Uhlenbeck] 
\label{thm:uhlenbeck}  Let $X$ be a compact K\"{a}hler manifold of complex
dimension $n$ and $(E,h)$ a hermitian vector bundle on $X$. Let $A_{i}$ be a
sequence of integrable unitary connections on $E$ with $\vert \Lambda_{\omega }F_{A_{i}}\vert$ uniformly bounded and $\Vert d_{A_i}^\ast
F_{A_{i}}\Vert _{L^{2}}\to 0$ . Fix $p>2n$. Then there is:

\begin{itemize}
\item a subsequence $($still denoted $A_{i}$$)$

\item a closed subset $Z^{an}\subset X$ of finite $(2n-4)$-Hausdorff measure;

\item a hermitian vector bundle $(E_{\infty },h_{\infty })$ defined on $X-
Z^{an}$ and $L^p_{2,loc.}$-isometric to $(E,h)$;

\item an admissible Yang-Mills connection $A_{\infty }$ on $E_{\infty }$;
\end{itemize}

such that up to unitary gauge equivalence $A_{i}$ converges weakly in $%
L_{1,loc}^{p}(X- Z^{an})$ to $A_{\infty }$.
\end{theorem}

We call the resulting connection $A_\infty$ an \emph{Uhlenbeck limit} and
the set $Z^{an}$ the \emph{analytic singular set}. A priori both depend on
the choice of subsequence in the statement of the theorem.

\begin{remark} \label{rem:removable}
By \cite[p.\ 62, Theorem]{Bando91}
(see also \cite{BandoSiu94}), the holomorphic structure $\bar\partial_{A_\infty}$ on 
$E_{\infty }$ extends as a reflexive sheaf $\mathcal{E }_{\infty }\to X$,
and the metrics extend smoothly to $X- \sing \mathcal{E }_{\infty }$. As
mentioned at the end of Section \ref{sec:stability}, since the limiting
connection is Yang-Mills, $\mathcal{E }_{\infty }$ decomposes
holomorphically and isometrically as a direct sum 
of stable reflexive sheaves with admissible Hermitian-Einstein metrics.
\end{remark}

The following result identifies Uhlenbeck limits of certain sequences of
isomorphic bundles.

\begin{theorem}[\protect\cite{DaskalWentworth04, Sibley}]
\label{thm:dws} Let $A_i$ be a sequence of connections in a complex gauge
orbit of a holomorphic bundle $\mathcal{E}$ with Harder-Narasimhan type $%
\vec\mu$ and satisfying the hypotheses of Theorem \ref{thm:uhlenbeck}.
Assume further that functionals $\mathrm{HYM}_\alpha(A_i)\to \mathrm{HYM}%
_\alpha(\vec\mu)$ for $\alpha\in [1,\infty)$ in a set that includes $2$ and
has a limit point. Then any Uhlenbeck limit of $A_i$ defines a reflexive
sheaf $\mathcal{E}_\infty$ which is isomorphic to $\Gr(\mathcal{E}%
)^{\ast\ast}$.
\end{theorem}

\noindent Here, the functionals $\mathrm{HYM}_\alpha$ are generalizations of
the Yang-Mills energy that were introduced in \cite{AtiyahBott82}. All the
hypotheses of Theorem \ref{thm:dws} are in particular satisfied if $A_i$ is
a sequence $A_{t_i}$, $t_i\to\infty$, along the Yang-Mills flow. In this
case, Hong and Tian \cite{HongTian04} prove that the convergence is in fact $%
C^{\infty }$ away from $Z^{ an}$.



We now give a more precise definition of the analytic singular set.  For a
sequence of connections $A_{i}$ satisfying the hypotheses of Theorem \ref%
{thm:uhlenbeck} and with Uhlenbeck limit $A_\infty$, define (cf.\ \cite[eq.\
(3.1.4)]{Tian00}), 
\begin{equation*}
Z^{an}=\bigcap_{\sigma_0\geq \sigma>0}\biggl\{x\in X: \liminf_{i\to
\infty}\sigma^{4-2n}\int_{B_{\sigma}(x)}\left\vert F_{A_i}\right\vert
^{2}dvol_{\omega } \geq \varepsilon _{0}\biggr\}.
\end{equation*}
The numbers $\varepsilon_0$, $\sigma_0$, are those that appear in the $%
\varepsilon$-regularity theorem \cite{Uhlenbeck82a, Uhlenbeck} and depend
only on the geometry of $X$ (see also Section \ref{sec:chernclass} below).
The density function is defined by taking a weak limit of the Yang-Mills
measure: 
\begin{equation*}
|F_{A_i}|^2(x)\, dvol_\omega\longrightarrow |F_{A_\infty}|^2(x)\,
dvol_\omega+\Theta(x)\mathcal{H}^{2n-4}\bigr\lfloor{Z^{an}_b} 
\end{equation*}
For almost all $x\in Z^{an}$ with respect to $(2n-4)$-Hausdorff measure $%
\mathcal{H}^{2n-4}$, 
\begin{equation*}
\Theta (x)=\lim_{\sigma\rightarrow 0}\lim_{i\rightarrow \infty
}\sigma^{4-2n}\int_{B_{\sigma}(x)}\left\vert F_{A_i}\right\vert
^{2}dvol_{\omega }
\end{equation*}%
The closed subset of $Z^{an}$ defined by 
\begin{equation*}
Z_{b}^{an}=\overline{\{x\in X\mid \Theta (x)>0,\lim_{\sigma\rightarrow
0}\sigma^{4-2n}\int_{B_{\sigma}(x)}\left\vert F_{A_{\infty }}\right\vert
^{2}dvol_{\omega }=0\}}
\end{equation*}%
is called the \emph{blow-up locus} of the sequence $A_{i}$.

As noted in Remark \ref{rem:removable}, $A_\infty$ decomposes into a direct sum of admissible Hermitian-Yang-Mills connections. Using \cite[Prop.\ 5.1.2]{Tian00}, it follows  from the removable singularities theorem of Tao and Tian \cite{TaoTian04} that there is a gauge tranformation $g $
on $X- Z^{an}$, such that $g(A_\infty)$ extends smoothly over the blow-up
locus. Moreover, it can be shown that $\mathcal{H}^{2n-4}(\overline{Z^{an}-
Z_{b}^{an})}=0$. Therefore, we can express 
\begin{equation}  \label{eqn:ansingset}
Z^{an}=Z_{b}^{an}\cup \sing A_{\infty}
\end{equation}
and $\sing A_{\infty }$ is the $\mathcal{H}^{2n-4}$-measure zero set where $%
A_\infty$ is singular (see \cite[p.\ 223]{Tian00}).


If the sequence $\{A_i\}$ happens to be a sequence of Hermitian-Yang-Mills
connections, or if it is a sequence along the Yang-Mills flow, then one can
say much more about the blow-up locus. Namely, we have the following
theorem, proven in the two different cases in \cite{Tian00} and \cite%
{HongTian04} respectively.

\begin{theorem}[Tian, Hong-Tian]
\label{thm:tian} If $A_{i}$ is either a sequence of Hermitian-Yang-Mills
connections, or a sequence $A_i=A_{t_i}$ of connections along the Yang-Mills
flow with $t_i\to \infty$, and $A_i$ has Uhlenbeck limit $A_\infty$, then
its blow-up locus $Z_{b}^{an}$ is a closed holomorphic subvariety of pure
codimension $2$. Furthermore, the density $\Theta $ is constant along each
of the irreducible codimension $2$ components $Z_{j}^{an}\subset Z^{an}$ and
there exist positive integers $m_{j}^{an}$ such that for any smooth $(2n-4)$%
-form $\Omega$, we have 
\begin{equation*}
\frac{1}{8\pi ^{2}}\lim_{i\rightarrow \infty }\int_{X} \Omega\wedge \Tr %
(F_{A_{i}}\wedge F_{A_{i}})=\frac{1}{8\pi ^{2}}\int_{X}\Omega\wedge\Tr%
(F_{A_{\infty }}\wedge F_{A_{\infty }})
+\sum_{j}m_{j}^{an}\int_{Z_{j}^{an}}\Omega .
\end{equation*}
\end{theorem}

\noindent It follows that $Z^{an}$ is an analytic subvariety. We will call
the numbers$\ m_{j}^{an}$, the \emph{analytic multiplicities}. For more
details see \cite{Tian00}, \cite{TianYang02}, and \cite{HongTian04}.

%
%

\section{A Singular Bott-Chern Formula}

\subsection{Statement of results}

Throughout this section we consider holomorphic vector bundles $\mathcal{E}$
on a compact K\"ahler manifold $(X,\omega)$ with a \emph{general} filtration %
\eqref{eqn:filtration} by saturated subsheaves. As above, let $\Gr(\mathcal{E%
})$ denote the graded sheaf associated to the filtration, and $\Gr(\mathcal{E%
})^{\ast\ast}$ its sheaf theoretic double dual. Then $Z^{alg}$ will refer to
the algebraic singular set defined in Section \ref{sec:algsingset}. The
codimension $2$ components $Z_j^{alg}\subset Z^{alg}$ have multiplicities $%
m_j^{alg}$ as in Definition \ref{def:algmult}.

There are two key steps in the proof of Theorem \ref{thm:bottchern}. The
first is the cohomological statement of the following result which will be
proved in Section \ref{sec:grr}.

\begin{proposition}
\label{prop:key} Let $X$ be a compact, complex manifold, and $\mathcal{T}\to
X$ a torsion sheaf. Assume that $\supp \mathcal{T}$ has codimension $p$.
Denote by $Z_{j}\subset \supp \mathcal{T}$ the irreducible components of
codimension exactly $p$ and by $[Z_j]$ their corresponding homology classes.
Then for all $k<p$, $ch_{k}(\mathcal{T})=0$, and $ch_{p}(\mathcal{T})=%
\mathrm{PD}(\sum_{j}m_{Z_{j}}\left[ Z_{j}\right] )$ in $H^{2p}(X, \mathbb{Q })$.%
\end{proposition}

\noindent In other words, the $p$-th component of the Chern character of a
torsion sheaf with support in codimension $p$ is the the cohomology class of
the cycle $\left[ \supp\mathcal{T}\right] _{p}$, and all lower components
vanish. This result is probably well known, but we have not been able to find
a proof in the literature. We therefore give a proof here which will 
use the subsequent discussion of the Grothendieck-Riemann-Roch theorem
for complex spaces.

Note that on a non-projective compact, complex manifold, a coherent sheaf
does not in general have a global resolution by locally free coherent
analytic sheaves. See the appendix of \cite{Voisin02} for an explicit
counter-example. Therefore, we cannot define Chern classes directly in this
way. However, after tensoring with the sheaf $\mathcal{A}_{X}$ of germs of
real analytic functions on $X$ there is a resolution by real analytic vector
bundles (cf.\ \cite{Grauert58}).

\begin{definition}
\label{def:chernchar} Let $X$ be a compact, complex manifold and $\mathcal{F}%
\to X$ a coherent sheaf. Choose a global resolution: $0\to E_{r}\to
E_{r-1}\to \cdots\to E_{0}\to \mathcal{F}\otimes _{\mathcal{O}_{X}}\mathcal{A%
}_{X}\to 0 $, where the $E_{i}$ are real analytic complex vector bundles on
X. Then define 
\begin{equation*}
ch_p(\mathcal{F})=\sum_{i=0}^r(-1)^i ch_p(E_i) 
\end{equation*}
\end{definition}

\noindent One can show that this definition does not depend on the choice of
global resolution (cf.\ \cite{BorelSerre58}). The Chern characters appearing
in Proposition \ref{prop:key} are defined in this way. For other approaches
to Chern classes of coherent sheaves, see \cite{Green, Grivaux}.

For a smooth hermitian metric $h_0$ on $\mathcal{E}$, let $A_0=(\mathcal{E},
h_0)$ denote the Chern connection. Then by Chern-Weil theory the smooth,
closed $(2,2)$-form: $ch_{2}(\mathcal{E},h_0) = -(1/8\pi ^{2}) \Tr %
(F_{A_0}\wedge F_{A_0}) $, represents the Chern character $ch_2(\mathcal{E})$
in cohomology. For an \emph{admissible} metric $h$ on $\Gr(\mathcal{E})$,
let 
\begin{equation}  \label{eqn:ch2}
ch_{2}(\Gr(\mathcal{E}),h) = -\frac{1}{8\pi ^{2}} \Tr (F_{A}\wedge F_{A})
\end{equation}
where $A$ is the Chern connection of $(\Gr(\mathcal{E}),h)$ on its locally
free locus $X-Z^{alg}$. Note that since $F_A\in L^2$, eq.\ \eqref{eqn:ch2}
defines  a $(2,2)$-current on $X$ by setting 
\begin{equation}  \label{eqn:current-def}
ch_{2}(\Gr(\mathcal{E}),h)(\Omega)=-\frac{1}{8\pi^2}\int_{X}\Omega\wedge \Tr%
(F_{A}\wedge F_{A})
\end{equation}%
for any smooth $(2n-4)$-form $\Omega $. The second step in the proof of
Theorem \ref{thm:bottchern}  is the following result which will be proved in
Section \ref{sec:chernclass} below.

\begin{proposition}
\label{prop:ch2} Let $\mathcal{E}$ and $\Gr(\mathcal{E})$ be as in the
statement of Theorem \ref{thm:bottchern}. Then for any admissible metric $h$
on $\Gr(\mathcal{E})$, the smooth form $ch_{2}(\Gr(\mathcal{E}),h)$  on $%
X-Z^{alg}$ defined by \eqref{eqn:ch2},  extends as a closed $(2,2)$-current
on $X$. Moreover, this current represents $ch_{2}(\Gr(\mathcal{E}%
)^{\ast\ast})$ in rational cohomology.
\end{proposition}

Assuming the two results above, we have the

\begin{proof}[Proof of Theorem \protect\ref{thm:bottchern}]
Consider the exact sequence%
\begin{equation*}
0\longrightarrow \Gr(\mathcal{E}) \longrightarrow \Gr(\mathcal{E})^{\ast
\ast }\longrightarrow \Gr(\mathcal{E})^{\ast \ast }/\Gr(\mathcal{E})
\longrightarrow 0.
\end{equation*}%
Then by the additivity of $ch_{2}$ over exact sequences we have 
\begin{equation*}
ch_{2}(\Gr(\mathcal{E})^{\ast \ast }/\Gr(\mathcal{E})) =ch_{2}(\Gr(\mathcal{E%
})^{\ast \ast })-ch_{2}(\Gr(\mathcal{E})) =ch_{2}(\Gr(\mathcal{E})^{\ast
\ast })-ch_{2}(\mathcal{E})\ , 
\end{equation*}
where we have used that $ch_{2}(\Gr(\mathcal{E}))=ch_{2}(\mathcal{E})$, by
the definition of $ch_{2}$ and additivity. Applying Proposition \ref%
{prop:ch2} to the right hand side, and  Proposition \ref{prop:key} to the
left hand side of the equation above we obtain 
\begin{equation*}
\frac{1}{8\pi ^{2}}\left( \Tr (F_{A_0}\wedge F_{A_0})-\Tr (F_{A}\wedge
F_{A})\right) =\sum\limits_{j}m_{j}^{alg}Z_{j}^{alg}+dd^{c}\Psi
\end{equation*}%
for some $(1,1)$-current $\Psi$ by the $dd^{c}$-lemma for currents. By
elliptic regularity, $\Psi$ may be taken to be smooth away from $\sing \Gr(\mathcal{E}%
)=Z^{alg}$. 
\end{proof}

We also note that Theorems \ref{thm:bottchern} and \ref{thm:tian} combine to
give the following

\begin{corollary}
\label{cor:cohomologous} The currents $\sum_j m_j^{an} Z_j^{an}$ and $\sum_k
m_k^{alg} Z_k^{alg}$ are cohomologous.
\end{corollary}

\subsection{Levy's Grothendieck-Riemann-Roch Theorem and the cycle $ch_{p}(%
\mathcal{T})$}

\label{sec:grr}

To prove Proposition \ref{prop:key} we recall a very general version of the
Riemann-Roch theorem for complex spaces. One may think of this theorem as
translating algebraic (holomorphic) data into topological data. It is
expressed in terms of $K$-theory, so we recall some basic definitions. For
the subsequent discussion $X$ will denote a compact, complex space. Note
that for such a space there is a topological embedding $X\hookrightarrow 
\CBbb^{N}$.

Let $K_{0}^{\mathrm{hol}}(X)$ denote the Grothendieck group of the category
of coherent analytic sheaves on $X$, that is, the free abelian group
generated by isomorphism classes of coherent sheaves modulo the relation
given by exact sequences. We will write $K_{0}^{\mathrm{top}}(X)$ for the
homology $K $-theory of the topological space underlying $X$. This is the
homology theory corresponding to the better known topological $K$-theory $K_{%
\mathrm{top}}^{0}(X)$ given by the Grothendieck group of the category of
topological vector bundles. The group $K_{0}^{\mathrm{top}}(X)$ may be
defined in this case by choosing an embedding $X\hookrightarrow \CBbb^{N}$ and declaring $K_{0}^{\mathrm{top}}(X)=K_{\mathrm{top}}^{0}(%
\CBbb^{N},\CBbb^{N}-X)$, where the group on the
right hand side is the usual relative $K$-theory (see for example \cite{BFM2}%
). For a proper map $f:X\longrightarrow Y $, we can also define a
pushforward map $f_{\ast }:K_{0}^{\mathrm{top}}(X)\rightarrow K_{0}^{\mathrm{%
top}}(Y)$ (see \cite{BFM2}), by factoring $f$ as an inclusion composed with
a projection.

With all of this understood, the version of the Grothendieck-Riemann-Roch
theorem proven by Levy \cite{Levy87} states that there is a natural
transformation of functors $\alpha $ from $K_{0}^{\mathrm{hol}}$ to $K_{0}^{%
\mathrm{top}}$. Explicitly, this means that for any two compact complex
spaces $X$ and $Y$, there are maps 
\begin{equation*}
\alpha _{X}:K_{0}^{\mathrm{hol}}(X)\longrightarrow K_{0}^{\mathrm{top} }(X)\
,\qquad \alpha _{Y}:K_{0}^{\mathrm{hol}}(Y)\longrightarrow K_{0}^{\mathrm{top%
} }(Y)
\end{equation*}%
such that for any proper morphism $f:X\rightarrow Y$ the following diagram
commutes:%
\begin{equation*}
\xymatrix{ K_{0}^{{\rm hol}}(X) \ar[r]^{\alpha _{X}} \ar[d]_{f_{!}} &
K_{0}^{{\rm top}}(X) \ar[d]^{f_{\ast } } \\ K_{0}^{{\rm hol}}(Y)
\ar[r]^{\alpha _{Y}} & K_{0}^{{\rm top}}(Y) }
\end{equation*}%
Here $f_{!}$ is Grothendieck's direct image homomorphism given by 
\begin{equation*}
f_{!}(\left[ \mathcal{F}\right] )=\sum_{i}(-1)^{i}\left[ R^{i}f_{\ast }%
\mathcal{F}\right] ,
\end{equation*}%
and $f_{\ast }$ is the pushforward map in $K$-theory. We also have the usual
Chern character%
\begin{equation*}
ch^{\ast }:K_{\mathrm{top}}^{0}(\CBbb^{N},\CBbb
^{N}-X)\longrightarrow H^{2\ast }(\CBbb^{N},\CBbb
^{N}-X,\QBbb ).
\end{equation*}%
We may furthermore define an homology Chern character: 
\begin{equation*}
ch_{\ast }:K_{0}^{\mathrm{top}}(X)\longrightarrow H_{2\ast }(X,\QBbb ),
\end{equation*}%
by taking a topological embedding $X\hookrightarrow \CBbb^{N}$
and then composing the maps:%
\begin{equation*}
K_{0}^{\mathrm{top}}(X)=K_{\mathrm{top}}^{0}(\CBbb^{N},\CBbb^{N}-X)\overset{ch^{\ast }}{\xrightarrow{\hspace*{.75cm}} }H^{2\ast }(%
\CBbb^{N},\CBbb^{N}-X,\QBbb )\cong
H_{2\ast }(X,\QBbb ),
\end{equation*}%
where the last isomorphism is Lefschetz duality.

If $X$ is nonsingular, so that Poincar\'{e} duality gives an isomorphism $%
K_{0}^{\mathrm{top}}(X)\cong K_{\mathrm{top}}^{0}(X)$, then the relationship
between the homology Chern character and the ordinary Chern character $%
ch^{\ast }:K_{\mathrm{top}}^{0}(X)\longrightarrow H^{2\ast }(X,\QBbb )$ is given by:%
\begin{equation*}
ch_{\ast }(\mathrm{PD}\,\eta )=\mathrm{PD}[ch^{\ast }(\eta )\cdot Td(X)].
\end{equation*}%
\ \ \ 

The homology Chern character is also a natural transformation with respect
to the pushforward maps. By composition there is a natural transformation of
functors 
\begin{equation*}
\tau =ch_{\ast }\circ \alpha :K_{0}^{\mathrm{hol}}(X)\rightarrow H_{2\ast
}(X,\mathbb{Q})
\end{equation*}
that satisfies the corresponding naturality property, i.e.\ for any proper
map of complex spaces $f:X\rightarrow Y$: $\tau \circ f_{!}=f_{\ast }\circ
\tau $. Here $f_{!}$ is as before, and $f_{\ast }$ is the usual induced map
in homology.

For a coherent sheaf $\mathcal{F}$, $\tau (\mathcal{F})$ is called the  
\emph{homology Todd class} of $\mathcal{F}$ and satisfies a number of
important properties.

$(1)$ If $\mathcal{F}$ is locally free and $X$ is smooth, then $\tau (%
\mathcal{F})$ is the Poincar\'e dual to 
\begin{equation*}
ch(\mathcal{F})\cdot Td(X)\in H^{2\ast }(X,\QBbb ),
\end{equation*}
and in particular $\tau (\mathcal{O}_{X})=\mathrm{PD}(Td(X))$.

$(2)$ For an embedding $\imath :X\hookrightarrow Y$, we have $\tau (\imath
_{\ast }\mathcal{F})=\imath _{\ast }(\tau (\mathcal{F}))$.

$(3)$ If the support $\supp\mathcal{F}$ has dimension $k$, then the
components of $\tau (\mathcal{F})$ in (real) dimensions $r>2k$ vanish.

$(4)$ For any exact sequence of sheaves: $0\to \mathcal{G}\to \mathcal{F}\to%
\mathcal{H }\to 0$, we have $\tau (\mathcal{F})=\tau (\mathcal{G})+\tau (%
\mathcal{H})$.

Property $(4)$ follows from \cite[Lemma 3.4(b)]{Levy87} together with the
additivity of the homology Chern character. For $\mathcal{F}$ locally free,
property $(1)$ is a restatement of of \cite[Lemma 3.4(c)]{Levy87} as
follows. The statement there is that for $\mathcal{F}$ locally free, $\alpha
(\mathcal{F})$ is Poincar\'{e} dual to the topological vector bundle
corresponding to $\mathcal{F}$ in $K_{\mathrm{top}}^{0}(X)$, then $(1)$
follows by applying $ch_{\ast }$. Property $(2)$ follows from the naturality
of $\tau $ and the fact that the higher direct images of an embedding
vanish. Property $3$ follows from dimensional considerations. Indeed, since $%
\supp\mathcal{F}$ is an analytic subvariety, it can be triangulated, and if
it has (real) dimension $2k$ the homology in higher dimensions will be zero.
Hence, $\tau $ evaluated on the restriction of $\mathcal{F}$ to $\supp%
\mathcal{F}$ has zero components in dimension $>2k$. This combined with
property $(2)$ applied to the embedding $\imath :\supp\mathcal{F}%
\hookrightarrow X$ proves $(3)$.

Now we have the following important consequence of these properties.

\begin{proposition}
\label{prop:components} For a coherent sheaf $\mathcal{F}$ on a complex
manifold, the component of $\tau (\mathcal{F})$ in degree $2\dim \supp 
\mathcal{F}$ is the homology class of the cycle%
\begin{equation*}
\sum_{j}\left( \rank _{\mathcal{O}_{Z_{j}}}\mathcal{F}\bigr|_{ Z_{j}}\right)
[Z_{j}]=\sum_{j}m_{Z_{j}}[Z_{j}]
\end{equation*}%
where the $Z_{j}$ are the irreducible components of $\supp \mathcal{F}$ of
top dimension. \ \ \ \ \ \ \ 
\end{proposition}

\begin{proof}
First, we prove that for a reduced and irreducible, closed complex subspace $%
Z $ of $X$, and a coherent sheaf $\mathcal{F}$ on $Z$, $\tau (\mathcal{F}%
)_{2\dim Z}=(\rank \mathcal{F)}[Z]$, where $[Z]$ denotes the generator of $%
H_{2\dim Z}(Z,\QBbb )$. Note that if $Z$ is smooth and $\mathcal{F%
}$ is locally free, the statement follows directly from property $(1)$. By
Hironaka's theorem we may resolve the singularities of $Z$ by a sequence of
blow-ups along smooth centers. By Hironaka's flattening theorem, we can
perform a further sequence of blow-ups along smooth centers to obtain a
complex manifold $\widehat{Z}$, and a map $\pi :\widehat{Z}\rightarrow Z$
such that 
\begin{equation*}
\pi :\widehat{Z}-\pi ^{-1}(\sing Z\cup \sing \mathcal{F})\rightarrow Z-(%
\sing Z\cup \sing \mathcal{F)}
\end{equation*}
is a biholomorphism, 
and $\mathcal{\widehat{F}}=\pi ^{\ast }\mathcal{F}/\mathrm{tor}(\pi ^{\ast }%
\mathcal{F})$ is locally free. Then $\tau (\mathcal{\widehat{F}})_{2\dim 
\widehat{Z}}=(\rank \mathcal{\widehat{F})}[\widehat{Z}]=(\rank \mathcal{F)}[%
\widehat{Z}]$. The fundamental class of an analytic subvariety is always
equal to the pushforward of the fundamental class of a resolution of
singularities (see \cite[Ch.\ 11.1.4]{Voisin02}), so we have that $\pi
_{\ast }(\tau (\mathcal{\widehat{F}})_{2\dim Z})=(\rank \mathcal{F)}[Z]$. By
naturality, this is also $\tau (\pi _{!}\mathcal{\widehat{F})}_{2\dim Z}$.
Since $\pi $ is, in particular, a proper map, the stalks of the higher
direct image sheaves are given by $(R^{i}\pi _{\ast }\mathcal{\widehat{F}}%
)_{z}=H^{i}(\pi ^{-1}(z),\mathcal{\widehat{F})}$, and so $R^{i}\pi _{\ast }%
\mathcal{\widehat{F}}$ is supported on a proper subvariety for $i>0$.
Therefore, $\tau (R^{i}\pi _{\ast }\mathcal{\widehat{F})}_{2\dim Z}=0$ for $%
i>0$ by property $(3)$. It follows that $\tau (\pi _{\ast }\mathcal{\widehat{%
F})}_{2\dim Z}\mathcal{=}(\rank \mathcal{F)}[Z]$. In fact, since $\mathrm{tor%
}(\pi ^{\ast }\mathcal{F})$ is supported on a divisor, by properties (1),
(3), and  $(4)$, and the above argument applied to the higher direct images
of $\pi ^{\ast }\mathcal{F}$, we have 
\begin{equation}  \label{eqn:rank}
(\rank \mathcal{F)}[Z]=\pi _{\ast }(\tau (\mathcal{\widehat{F})}_{2\dim Z})%
\mathcal{=\pi }_{\ast }(\mathcal{\tau (}\pi ^{\ast }\mathcal{F})_{2\dim
Z})=\tau (\pi _{\ast }\pi ^{\ast }\mathcal{F)}_{2\dim Z}.
\end{equation}
On the other hand the natural map $\mathcal{F}\overset{\alpha}{%
\longrightarrow}\pi _{\ast }\pi ^{\ast }\mathcal{F}$ is an isomorphism away
from the proper subvariety $\sing Z \cup\, \sing \mathcal{F}$. Therefore,
the sheaves $\ker (\alpha )$ and the quotient $Q$ of $\pi _{\ast }\pi ^{\ast
}\mathcal{F}$ by $\mathcal{F}/\ker (\alpha )$ are supported on $\sing Z\cup %
\sing \mathcal{F}$. Considering the exact sequence:%
\begin{equation*}
0\longrightarrow \mathcal{F}/\ker (\alpha )\longrightarrow \mathcal{\pi }_{\ast }\pi ^{\ast }\mathcal{F}\longrightarrow Q\longrightarrow 0,
\end{equation*}
we see by $(4)$ that 
\begin{equation}  \label{eqn:tau-exact}
\tau (\mathcal{\pi }_{\ast }\pi ^{\ast }\mathcal{F})=\tau (\mathcal{F}/\ker
(\alpha ))+\tau (Q)=\tau (\mathcal{F})-\tau (\ker (\alpha ))+\tau (Q)
\end{equation}
Since $\ker (\alpha )$ and $Q$ are supported on $Z_{\sing }$, $(3)$ implies $%
\tau (\ker (\alpha ))_{2\dim Z}=$ $\tau (Q)_{2\dim Z}=0 $. Therefore, taking
the top dimensional component in \eqref{eqn:tau-exact} we obtain: 
\begin{equation*}
\tau (\mathcal{F})_{2\dim Z}=\tau (\mathcal{\pi }_{\ast }\pi ^{\ast }%
\mathcal{F})_{2\dim Z},
\end{equation*}
so by \eqref{eqn:rank}, $\tau (\mathcal{F})_{2\dim Z}=(\rank \mathcal{F)}[Z] 
$.

Now we prove the proposition. If $\supp \mathcal{F}=Z\overset{\imath }{%
\hookrightarrow }X$ is irreducible so that $\imath ^{\ast }\mathcal{F}$ has
constant rank on $Z$, then by what we have just proven $\tau (\mathcal{F}%
\bigr|_{ Z})_{2\dim Z} =(\rank \mathcal{F}\bigr|_{ Z})[Z]$. By naturality,
and the fact that $\imath _{\ast }\imath ^{\ast }\mathcal{F}=\mathcal{F}$,
we have $\tau (\mathcal{F})_{2\dim Z}=$ $(\rank \mathcal{F}\bigr|_{ Z})[Z]$.
If $Z$ has several irreducible components $Z_{1},\cdots ,Z_{l}$, of
dimension $\dim Z$ then we may similarly consider the embeddings $\imath
_{j}:Z_{j}\hookrightarrow X$. Then $\rank \mathcal{F}\bigr|_{ Z_j}=\rank %
((\imath_j)_\ast \mathcal{F}\bigr|_{ Z_j})$, and this latter sheaf is
supported on $Z_{j}$, so that $\tau ((\imath_j)_\ast \mathcal{F}\bigr|_{
Z_j})_{2\dim Z}=(\rank \mathcal{F}\bigr|_{ Z_j})[Z_{j}]$. Moreover, the
natural map $\displaystyle \mathcal{F}\rightarrow
\bigoplus_{j}(\imath_j)_\ast \mathcal{F}\bigr|_{ Z_j}$ is an injection, and
the quotient is supported on the pairwise intersections of the $Z_{j}$ (and
the irreducible components of lower dimension), and so has zero $\tau $ in
dimension $2\dim Z$. Therefore by properties $(3)$ and $(4)$ we have 
\begin{equation*}
\tau (\mathcal{F})_{2\dim Z}=\tau \biggl(\bigoplus_{j}(\imath_j)_\ast 
\mathcal{F}\bigr|_{ Z_j}\biggr)_{2\dim Z}=\sum_{j}(\rank \mathcal{F}\bigr|_{
Z_j})[Z_{j}].
\end{equation*}
\end{proof}

\begin{remark}
If $X$ and $Y$ are compact, complex manifolds, Atiyah and Hirzebruch \cite%
{AtiyahHirzebruch61} prove the topological Grothendieck-Riemann-Roch
theorem.\ That is, for any continuous map $f:X\rightarrow Y$, the diagram:%
\begin{equation*}
\xymatrix{ K_{{\rm top}}^{0}(X) \ar[r]^{ch^\ast\ }\ar[d]_{f_\ast } &
H^{2\ast }(X, \QBbb ) \ar[d]^{f_\ast } \\ K_{{\rm top}}^{0}(Y)
\ar[r]^{ch^{\ast }\ } & H^{2\ast }(Y,\QBbb) }
\end{equation*}%
commutes up to multiplication by Todd classes on both sides, where $f_{\ast
} $ on both sides of the diagram is the Gysin map given by the induced map
in homology and Poincar\'{e} duality. Combining this with Levy's theorem, it
follows that for a proper holomorphic map $f:X\rightarrow Y$ and any class $%
\eta \in K_{0}^{\mathrm{hol}}(X)$, 
\begin{equation*}
f_{\ast }(ch(\eta )\cdot Td(X))=ch(f_{!}(\eta ))\cdot Td(Y)
\end{equation*}%
in $H^{2\ast }(Y,\QBbb )$. This is the exact analogue of the
formula proven in \cite{OTT}, where the identity is in the Hodge ring rather
than rational cohomology.
\end{remark}

\begin{remark}
\label{rmk:k} Define $K_{\mathrm{hol}}^{0}(X)$ to be the Grothendieck group
of holomorphic vector bundles on a complex manifold $X$. Let $\mathcal{F}$
be a coherent analytic sheaf on $X$ defining a class $\eta \in K_{0}^{%
\mathrm{hol }}(X)$. By the lemma in Fulton \cite{Fulton77}, there is a
complex manifold $\widehat X$ with a proper morphism $\pi :\widehat{X}%
\longrightarrow X$ an element $\zeta \in K_{\mathrm{hol}}^{0}(\widehat{X})$
such that $\pi _{!}(\mathrm{PD}\,\zeta )=\eta $, (where $\mathrm{PD}\, \zeta 
$ is given by the class $\zeta \otimes \lbrack \mathcal{O}_{\widehat{X}}]$,
the cap product with the fundamental class in $K_{0}^{\mathrm{hol}}(\widehat{%
X})$). Then by property $(1)$, naturality of $\tau $, and the previous
remark: 
\begin{eqnarray*}
\tau (\mathcal{F}) &=&\tau (\eta )=\pi _{\ast }(\tau (\mathrm{PD}\,\zeta
)=\pi _{\ast }[(ch^{\ast }\mathrm{PD}\,\zeta \cdot Td(\widehat{X})]\cap
\lbrack \widehat{X}]) \\
&=&[ch(\mathcal{\pi }_{!}(\mathrm{PD}\,\zeta ))\cdot Td(X)]\cap \lbrack
X]=ch(\mathcal{F})\cdot Td(X)\cap \lbrack X].
\end{eqnarray*}
Therefore, in fact property $(1)$ holds for arbitrary coherent sheaves.
\end{remark}

Finally, we have the

\begin{proof}[Proof of Proposition \protect\ref{prop:key}]
If $\mathcal{T}$ is a torsion sheaf on a complex manifold $X$, then $Z=\supp 
\mathcal{T}$ \ is a complex analytic subvariety of $X$ (via the annihilator
ideal sheaf). Suppose $\codim Z=p$. Then for any $k<p$, by property $(3)$ we
have $\tau (\mathcal{T})_{\dim 2n-2k}=0$. Applying Remark \ref{rmk:k} we
have: 
\begin{equation*}
\tau (\mathcal{T})_{\dim 2n-2}=\mathrm{PD}\left[ c_{1}(\mathcal{T})+\frac{%
\rank\mathcal{T}}{2}c_{1}(X)\right] =\mathrm{PD}[c_{1}(\mathcal{T})]
\end{equation*}%
since $\mathcal{T}$ is torsion and therefore has generic rank $0$. If $p\geq
2$, the left hand side vanishes. By induction we therefore see that for any $%
k<p$, we have: 
\begin{equation*}
0=\tau (\mathcal{T})_{\dim 2n-2k}=\mathrm{PD}\left[\sum_{i\leq k}ch_{i}(%
\mathcal{T})\cdot Td_{k-i}(X)\right]=\mathrm{PD}[ch_{k}(\mathcal{T})].
\end{equation*}%
By Proposition \ref{prop:components}, $\tau (\mathcal{T})_{\dim
2n-2p}=\sum_{j}m_{Z_{j}}[Z_{j}]$, and by Remark \ref{rmk:k} it is also $[ch(%
\mathcal{T})\cdot Td(X)]_{p}\cap \lbrack X]$. Since all components of the
Chern character of degree less than $p$ are zero, this is equal to $\mathrm{%
PD}[ch_{p}(\mathcal{T})]$, so we have: $ch_{p}(\mathcal{T})=\mathrm{PD}%
\bigl( \sum_{j}m_{Z_{j}}[Z_{j}]\bigr)
$.
\end{proof}

\subsection{Proof of Proposition \protect\ref{prop:ch2}}

\label{sec:chernclass}

In this section we show that the second Chern character form of an
admissible connection on a reflexive sheaf actually represents the
cohomology class of the second Chern character, at least when the sheaf
satisfies a certain technical, topological assumption. The proof of this
result follows the general argument in \cite{Tian00}. However, instead of an
admissible Yang-Mills connection and the corresponding monotonicity formula
of Price which are the context of \cite{Tian00}, we have the Chern
connection of an admissible metric in the sense of Section \ref%
{sec:stability}. The uniform bound on the Hermitian-Einstein tensor means we
may instead use the monotonicity formula and $L^p$-estimates derived for
integrable connections on K\"ahler manifolds in \cite{Uhlenbeck}. For
completeness, we provide details of the proof below.

Let us review the two key results of \cite{Uhlenbeck} that we will
need. First, let $E\to X$ be a hermitian bundle with an integrable connection $A$,
and suppose $\sup_X |\Lambda_\omega F_A|\leq H_0$. For constants $C_1$, $C_2$, and $%
x\in X$, $\sigma>0$, define 
\begin{equation}  \label{eqn:density}
e_A(x,\sigma)=C_1\sigma^4H_0^2+ (1+C_2
\sigma^2)^{2n-2}\sigma^{4-2n}\int_{B_\sigma(x)}|F_A|^2 dvol_\omega
\end{equation}
Then for appropriately chosen $C_1$, $C_2$, $\sigma_0>0$ (depending only on
the geometry of $X$), it follows from \cite[Thm.\ 3.5]{Uhlenbeck} that $%
e_A(x,\sigma)$ is monotone increasing, i.e.\ for all $0<\sigma\leq\rho\leq
\sigma_0$, 
\begin{equation}  \label{eqn:monotone}
e_A(x,\sigma)\leq e_A(x,\rho)
\end{equation}
Second, 
 there are constants $\varepsilon_0>0$, $C>0$
(depending only on the geometry of $X$) such that if $4\sigma\leq
\sigma_0$ and $e_A(x,4\sigma)<\varepsilon_0$, then one can find a gauge
transformation $g$ so that: 
\begin{equation}  \label{eqn:supA}
\sup_{B_{\sigma/2}(x)}\left|g(A)\right|
\leq
\frac{C}{\sigma}\left\{ \biggl(\sigma^{4-2n}\int_{B_{4\sigma}(x)}|F_A|^2 dvol_\omega\biggr)%
^{1/2} +\sigma^2 H_0 \right\}
\end{equation}
Indeed, by \cite[Thm.\ 2.6]{Uhlenbeck} the conditions guarantee that the
hypotheses of \cite[Thm.\ 2.2]{Uhlenbeck} are satisfied locally about $x$.
We refer to the proof of that theorem,  which uses  the main result of  
\cite[Sec.\ 2]{Uhlenbeck82a}. A consequence is that we can find a gauge transformation
such that, in a local trivialization and after rescaling to the unit ball, 
$d^\ast A=0$, and 
\begin{equation} \label{eqn:A}
\Vert A\Vert_{L^p_1(B_{1/2}(x))}
\leq C_p\left\{ \biggl(\int_{B_{1}(x)}|F_A|^2 dvol_\omega\biggr)%
^{1/2} + \widehat H_0 \right\}
\end{equation}
for any $2< p< 2n$, and $\widehat H_0=\sigma^2 H_0$.
In particular, $A\in L^q$ with estimates for any $q$.  Now, the fact that $A$ is
integrable and in good gauge implies (see \cite[eq.\ (27)]{Uhlenbeck}),
\begin{align}
\begin{split} \label{eqn:elliptic}
\dbar \alpha &= -[\alpha, \alpha]\\
\dbar^\ast\alpha &= -\sqrt{-1}\Lambda
[\alpha,\alpha^\ast]+\sqrt{-1}\Lambda F_A 
\end{split}
\end{align}
where $\alpha=A^{0,1}$, $\alpha^\ast=-A^{1,0}$. Applying the elliptic
estimate for $\dbar$ to \eqref{eqn:elliptic} (recall that $\vert \sqrt{-1}\Lambda
F_A\vert\leq \widehat H_0$), we see that \eqref{eqn:A} holds for some
$p> 2n$.  We therefore obtain a H\"older 
bound on $A$, and \eqref{eqn:supA} follows after rescaling back to
$B_\sigma(x)$.

With these preliminaries we give the

\begin{proof}[Proof of Proposition \protect\ref{prop:ch2}]
Let $A$ denote the Chern connection of $(\Gr(\mathcal{E}),h)$.  We wish to
show the current defined in \eqref{eqn:current-def} is closed, i.e.\ $ch_{2}(%
\Gr(\mathcal{E}),h)(\Omega)=0$ for any $\Omega=d\phi$.  The argument closely
follows the proof in \cite[Prop.\ 2.3.1]{Tian00}. By using a partition of
unity we may assume $\phi$ is compactly supported in a coordinate ball $U$.
Note that $\Gr(\mathcal{E})$ is smoothly isomorphic to the underlying vector
bundle $E$ of $\mathcal{E}$ on $X-Z^{alg}$. Hence, we may assume that $A$ is
an integrable connection on a trivial bundle on $U$, smooth away from $Z$,
with $F_A\in L^2$ and $\vert\Lambda_\omega F_A\vert$ uniformly bounded by $H_0$.
For $0<\sigma<\sigma_0$, $0<\varepsilon<\varepsilon_0$, let 
\begin{equation*}
E_{\sigma,\varepsilon}=\left\{ x\in U :
e_A(x,4\sigma)\geq\varepsilon\right\} 
\end{equation*}
where $e_A(x,\sigma)$ is defined in \eqref{eqn:density}. Denote $\sigma$%
-neighborhoods of subsets of $U$ by $\mathcal{N}_\sigma$.

Choose a covering $\{ B_{2\sigma}(x_k), B_{2\sigma}(y_k)\}$ of $Z\cup
E_{\sigma,\varepsilon}$, with $x_k\in Z$, $y_k\in E_{\sigma, \varepsilon}$,
and such that the balls $B_{\sigma}(x_k)$, $B_{\sigma}(y_k)$, are all
disjoint. If $x\not \in \bigcup_k B_{6\sigma}(x_k)\cup B_{2\sigma}(y_k)$,
then $B_{4\sigma}(x)\subset X-Z$, and $e_A(x,4\sigma)<\varepsilon$. In
particular, for any $x\not \in \mathcal{N}_{8\sigma}(Z)\cup\mathcal{N}%
_{4\sigma}(E_{\sigma,\varepsilon})$, there is a gauge transformation $g$
such that \eqref{eqn:supA} holds. As in \cite[pp.\ 217-218]{Tian00}, we can
piece together the gauge transformations to obtain a global Chern-Simons
form away from $\mathcal{N}_{8\sigma}(Z)\cup\mathcal{N}_{4\sigma}(E_{\sigma,%
\varepsilon})$: 
\begin{equation*}
CS(A)=\tr\bigl(A\wedge F_{A} +(1/3)A\wedge A\wedge A\bigr)
\end{equation*}
for $A$ in this gauge, with $dCS(A) =\tr (F_A\wedge F_A) $. Now for $x\in
B_{16\sigma}(x_k)-B_{8\sigma}(x_k)$, and using
\eqref{eqn:supA}, 
\begin{align}
\left|CS(A)(x)\right|&\leq |A(x)||F_A(x)|+(1/3)|A(x)|^3  \notag \\
&\leq \frac{1}{2\sigma}|A(x)|^2+\frac{\sigma}{2}|F_A(x)|^2+\frac{1}{3}%
|A(x)|^3  \notag \\
&\leq C \sigma^{1-2n}\int_{B_{4\sigma}(x)}|F_A|^2 dvol_\omega +C\sigma H_0^2
+\frac{\sigma}{2}|F_A(x)|^2  \notag \\
&\leq C\sigma^{1-2n}\int_{B_{20\sigma}(x_k)}|F_A|^2 dvol_\omega + C\sigma
H_0^2 +\frac{\sigma}{2}|F_A(x)|^2  \label{eqn:csestx}
\end{align}
Similarly, for $y\in B_{8\sigma}(y_k)-B_{4\sigma}(y_k)$, 
\begin{equation}  \label{eqn:csesty}
\left|CS(A)(y)\right|\leq C\sigma^{1-2n}\int_{B_{12\sigma}(y_k)}|F_A|^2
dvol_\omega + C\sigma H_0^2 +\frac{\sigma}{2}|F_A(y)|^2
\end{equation}
(we have assumed $\varepsilon\leq 1$). Now choose a smooth cut-off function $%
\eta$, $\eta(t)\equiv 0$ for $t\leq 1$, $\eta(t)\equiv 1$ for $t\geq 2$. It
follows that 
\begin{align}
\left| \int_X d\phi\wedge \Tr(F_A\wedge F_A) \right| &=\lim_{\sigma\to 0}
\left| \int_{X}\eta\left(\mathrm{dist}(x,Z)/8\sigma\right)\eta\left(\mathrm{%
dist}(x,E_{\sigma,\varepsilon})/4\sigma\right) d\phi\wedge dCS(A)\right| 
\notag \\
&\leq \lim_{\sigma\to 0}\biggl\{ \int_{8\sigma\leq \mathrm{dist}(x,Z)\leq
16\sigma}\frac{1}{8\sigma} |d\phi||CS(A)| dvol_\omega(x)  \notag \\
&\qquad\qquad\qquad + \int_{4\sigma\leq \mathrm{dist}(x,E_{\sigma,%
\varepsilon})\leq 8\sigma}\frac{1}{4\sigma} |d\phi||CS(A)| dvol_\omega(x)%
\biggr\}  \notag \\
& \leq C\sup|d\phi| \lim_{\sigma\to 0} \sum_k\Bigl\{ \int_{B_{20%
\sigma}(x_k)}(\left| F_A\right|^2 +CH_0^2) dvol_\omega  \notag \\
&\qquad\qquad\qquad+\int_{B_{12\sigma}(y_k)}(\left| F_A\right|^2 +CH_0^2)
dvol_\omega\Bigr\}  \notag \\
& \leq C\sup|d\phi| \lim_{\sigma\to 0} \int_{\mathcal{N}_{20\sigma}(Z)\,\cup%
\,\mathcal{N}_{12\sigma}(E_{\sigma,\varepsilon})}(\left| F_A\right|^2
+CH_0^2)dvol_\omega  \label{eqn:cslim}
\end{align}
since the number of $i,j,k,l$ such that the balls $B_{20\sigma}(x_i)$, $%
B_{20\sigma}(x_j)$, $B_{12\sigma}(y_k)$, $B_{12\sigma}(y_l)$, intersect is
bounded independently of $\sigma$.

\begin{claim}
For $0<\sigma^{\prime }\leq \sigma$, we have the following inclusions: 
\begin{equation*}
\mathcal{N}_{20\sigma^{\prime }}(Z)\cup \mathcal{N}_{12\sigma^{\prime
}}(E_{\sigma^{\prime },\varepsilon}) \subset \mathcal{N}_{20\sigma}(Z)\cup 
\mathcal{N}_{12\sigma}(E_{\sigma,\varepsilon}) 
\end{equation*}
\begin{equation*}
\bigcap_{\sigma>0}\left\{\mathcal{N}_{20\sigma}(Z)\cup \mathcal{N}%
_{12\sigma}(E_{\sigma,\varepsilon})\right\}\subset Z 
\end{equation*}
\end{claim}

\noindent Indeed, if $y\in \mathcal{N}_{12\sigma^{\prime
}}(E_{\sigma^{\prime },\varepsilon})$ and $y\not\in \mathcal{N}_{20\sigma}(Z)
$, then there is $x\in E_{\sigma^{\prime },\varepsilon}$ such that $%
d(x,y)<12\sigma^{\prime }\leq 12\sigma$, and if $z\in Z$, then $20\sigma\leq
d(y,z)\leq d(x,y)+d(x,z)< d(x,z)+12\sigma$, so $B_{4\sigma}(x)\subset X-Z$.
Now \eqref{eqn:monotone} applies, and so $\varepsilon\leq
e_A(x,4\sigma^{\prime })\leq e_A(x,4\sigma)$. It follows that $x\in
E_{\sigma,\varepsilon}$ and $y\in \mathcal{N}_{12\sigma}(E_{\sigma,%
\varepsilon})$. This proves the first statement in the claim. The second
statement follows from the fact that $A$ is smooth away from $Z^{alg}$;
hence, $\displaystyle\lim_{\sigma\to 0} e_A(x,\sigma)=0$ for $x\not\in
Z^{alg}$.  

Now by the claim and the dominated convergence theorem, the limit
on the right hand side of \eqref{eqn:cslim} vanishes, and closedness of $%
ch_{2}(\Gr(\mathcal{E}),h)$ follows.

Since $ch_{2}(\Gr(\mathcal{E}),h)$ is a closed current it defines a
cohomology class. By Poincar\'e duality, to check that indeed $[ch_{2}(\Gr(%
\mathcal{E}),h)]=ch_{2}(\Gr(\mathcal{E})^{\ast\ast})$, it suffices to show
that for any smooth, closed $(2n-4)$-form $\Omega$ whose cohomology class is
dual to a $4$-dimensional rational homology class $[\Sigma]$, 
\begin{equation}  \label{eqn:evaluate}
ch_{2}(\Gr(\mathcal{E})^{\ast\ast})[\Sigma]=ch_{2}(\Gr(\mathcal{E}%
),h)(\Omega)
\end{equation}
Since a multiple of a rational homology class is represented by an embedded
manifold, and since $\Gr(\mathcal{E})^{\ast\ast}$ is locally free away from
a set of (real) codimension $\geq 6$, by a transversality argument we may
assume (after passing to an integer multiple) that the homology class $%
[\Sigma]$ is represented by a smoothly embedded submanifold $\Sigma\subset X-%
\sing\Gr(\mathcal{E})^{\ast\ast}$. By the Thom isomorphism we may then
choose the form $\Omega$ to be compactly supported in $X-\sing\Gr(\mathcal{E}%
)^{\ast\ast}$. Find a global resolution 
\begin{equation}  \label{eqn:resolution}
0\longrightarrow E_{r}\longrightarrow E_{r-1}\longrightarrow
\cdots\longrightarrow E_{0}\longrightarrow \Gr(\mathcal{E}%
)^{\ast\ast}\otimes _{\mathcal{O}_{X}}\mathcal{A}_{X}\longrightarrow 0
\end{equation}
where the $E_{i}$ are real analytic complex vector bundles on X, and fix
smooth connections $\nabla_i$ on $E_i$. Then by Definition \ref%
{def:chernchar}, 
\begin{equation}  \label{eqn:ch2-rep1}
ch_{2}(\Gr(\mathcal{E})^{\ast\ast}) =\left[ -\frac{1}{8\pi^2}\sum_{i=0}^r
(-1)^i \tr(F_{\nabla_i}\wedge F_{\nabla_i}) \right]
\end{equation}
If we choose a sequence of smooth hermitian metrics $h^{\ast\ast}_j$ on the
locally free part of $\Gr(\mathcal{E})^{\ast\ast}$ with Chern connections $%
A_j^{\ast\ast}$, then since \eqref{eqn:resolution} is an exact sequence of
analytic vector bundles away from $\sing\Gr(\mathcal{E})^{\ast\ast}$, there
are smooth forms $\Psi_j$ such that 
\begin{equation}  \label{eqn:ch2-rep2}
-\frac{1}{8\pi^2} \tr(F_{A_j^{\ast\ast}}\wedge F_{A_j^{\ast\ast}}) +\frac{1}{%
8\pi^2}\sum_{i=0}^r (-1)^i \tr(F_{\nabla_i}\wedge F_{\nabla_i}) =d\Psi_j
\end{equation}
on $X-\sing\Gr(\mathcal{E})^{\ast\ast}$. Finally, by Theorem \ref%
{thm:bando-siu} (ii), we may arrange that $h^{\ast\ast}_j\to h$ in $%
L^p_{2,loc.}$ for some $p> 2n$. Then $h_j^{\ast\ast}$ and $%
(h_j^{\ast\ast})^{-1}$ are uniformly bounded on compact subsets of $X-\sing%
\Gr(\mathcal{E})^{\ast\ast}$, and it follows that $F_{A_j^{\ast\ast}}\to F_A$
in $L^2_{loc.}$.  Then for $\Omega$ as above, we obtain from %
\eqref{eqn:ch2-rep1} and \eqref{eqn:ch2-rep2} that 
\begin{align*}
ch_2(\Gr(\mathcal{E})^{\ast\ast})[\Sigma] &= -\frac{1}{8\pi^2}
\int_X\Omega\wedge\sum_{i=0}^r (-1)^i \tr(F_{\nabla_i}\wedge F_{\nabla_i}) \\
&=-\int_X \Omega\wedge\left( \frac{1}{8\pi^2} \tr(F_{A_j^{\ast\ast}}\wedge
F_{A_j^{\ast\ast}})+d\Psi_j\right) \\
&=-\frac{1}{8\pi^2}\int_X \Omega\wedge \tr(F_{A_j^{\ast\ast}}\wedge
F_{A_j^{\ast\ast}})\qquad\text{for all $j$} \\
&=-\frac{1}{8\pi^2}\lim_{j\to \infty}\int_X \Omega\wedge 
\tr(F_{A_j^{\ast\ast}}\wedge F_{A_j^{\ast\ast}}) \\
&=-\frac{1}{8\pi^2}\int_X \Omega\wedge \tr(F_{A}\wedge F_{A})
\end{align*}
Hence, from the definition \eqref{eqn:current-def}, eq.\ \eqref{eqn:evaluate}
holds, and this completes the proof of the proposition. 
\end{proof}

\section{Comparison of singular sets}

\label{Section4}

\subsection{A slicing lemma}

Let $z$ be a smooth point of a codimension $2$ subvariety $Z\subset X$. We
say that $\Sigma $ is a \emph{transverse slice} to $Z$ at $z$ if $\Sigma
\cap Z=\{z\}$ and $\Sigma $ is the intersection of a linear subspace $%
\mathbb{C}^{2}\hookrightarrow \mathbb{C}^{n}$ in some coordinate ball
centered at $z$ that is transverse to $T_{z}Z$ at the origin. Suppose that $T
$ is a smooth, closed $(2,2)$ form and that we have an equation 
\begin{equation*}
T=mZ+dd^{c}\Psi 
\end{equation*}%
where $\Psi $ is a $(1,1)$-current, smooth away from $Z$,  $mZ$ is the current of integration over
the nonsingular points of $Z$ with multiplicity $m$, and the equation is
taken in the sense of distributions. Then for a transverse slice, 
\begin{equation}
m=\int_{\Sigma }T-\int_{\partial \Sigma }d^{c}\Psi   \label{eqn:transverse}
\end{equation}%
Indeed, if we choose local coordinates so that a neighborhood of $z$ is
biholomorphic to a polydisk $\Delta \in \mathbb{C}^{n}$, and $Z\cap \Delta $
is given by the coordinate plane $z_{1}=z_{2}=0$, then by King's formula
(cf.\ \cite[Ch.\ III, 8.18]{Demailly}) we have an equation of
currents on $\Delta $: 
\begin{equation*}
T=dd^{c}\left[ \Psi +mudd^{c}u\right] 
\end{equation*}%
where $u(z)=(1/2)\log (|z_{1}|^{2}+|z_{2}|^{2})$ (here, $-\pi i dd^c=\partial\dbar$). Write $T=dd^{c}\alpha $
for a smooth form $\alpha $ on $\Delta $. By the regularity theorem and the
Poincar\'{e} lemma \cite{deRham}, we can find a  form $\beta $, smooth away from $Z\cap\Delta$,  such that 
\begin{equation*}
d^{c}\left[ \Psi -\alpha +mudd^{c}u\right] =d\beta 
\end{equation*}%
on $\Delta $. Let $\Delta _{\varepsilon }=\{z\in \Delta :|z|\leq \varepsilon
\}$. It follows that: 
\begin{align*}
\int_{\Sigma }T& =\lim_{\varepsilon \rightarrow 0}\int_{\Sigma \cap \Delta
_{\varepsilon }^{c}}T=\lim_{\varepsilon \rightarrow 0}\int_{\Sigma \cap
\Delta _{\varepsilon }^{c}}dd^{c}\Psi =\int_{\partial \Sigma }d^{c}\Psi
-\lim_{\varepsilon \rightarrow 0}\int_{\Sigma \cap \partial \Delta
_{\varepsilon }}d^{c}\Psi  \\
& =\int_{\partial \Sigma }d^{c}\Psi -\lim_{\varepsilon \rightarrow
0}\int_{\Sigma \cap \partial \Delta _{\varepsilon }}d^{c}\alpha
+m\lim_{\varepsilon \rightarrow 0}\int_{\Sigma \cap \partial \Delta
_{\varepsilon }}d^{c}(udd^{c}u) \\
& =\int_{\partial \Sigma }d^{c}\Psi +m
\end{align*}%
by direct computation. The next result shows that the analytic
multiplicities may also be calculated by restricting to transverse slices.

\begin{lemma}
\label{lem:slice} Let $A_i$ be as in Theorem \ref{thm:tian} and $Z\subset
Z_{b}^{an}$  an irreducible component of the blow-up set. For a  transverse
slice $\Sigma $ at a smooth point $z\in Z$, we have: 
\begin{equation*}
m_{Z}^{an}=\lim_{i\rightarrow \infty } \frac{1}{8\pi ^{2}}\int_{\Sigma
}\left\{\Tr (F_{A_i}\wedge F_{A_i})- \Tr (F_{A_{\infty }}\wedge F_{A_{\infty
}})\right\} .
\end{equation*}
\end{lemma}

\begin{proof}
Assume without loss of generality that $\Sigma\subset B_\sigma(z)$, where
the ball is chosen so that $A_i\to A_\infty$ smoothly and uniformly on
compact subsets of $B_{2\sigma}(z)-Z$. We furthermore assume the exponential
map $\exp_z$ at $z$ defines a diffeomorphism onto $B_{2\sigma}(z)$, and that 
$Z\cap B_{2\sigma}(z)$ is a submanifold. For $\lambda>0$,  let $A_{i,\lambda}
$ be the connection on $T_zX$ obtained by pulling back $A_i$ by the
exponential map, followed by the rescaling $v\mapsto \lambda v$ (cf.\ \cite[%
Sec.\ 3]{Tian00}). Then by definition of the blow-up connection and the
uniqueness of tangent cones, there is a sequence $\lambda_i\downarrow 0$
such that 
\begin{equation*}
m_Z^{an}= \lim_{i\to \infty}\frac{1}{8\pi^2}\int_{V^\perp\cap\, \mathbf{B}%
_1(0)}\left\{ \tr (F_{A_{i,\lambda_i}}\wedge F_{A_{i,\lambda_i}})- \tr %
(F_{A_{\infty,\lambda_i}}\wedge F_{A_{\infty,\lambda_i}}) \right\} 
\end{equation*}
(see \cite[eq.\ (4.2.7)]{Tian00} and \cite[eq.\ (5.5)]{HongTian04}). The
notation $V^\perp$ denotes the orthogonal complement of $V=T_zZ\subset T_zX$%
, and $\mathbf{B}_1(0)$ is the unit ball about the origin. 
Let $S_{\lambda_i}=\exp_z(\lambda_i (V^\perp\cap\, \mathbf{B}_1(0)))$. For
sufficiently large $i$ (i.e.\ $0<\lambda_i$ small), we may assume $%
S_{\lambda_i}\subset B_\sigma(z)$. At this point we choose smooth maps $u_i:
V^\perp\cap\, \mathbf{B}_1(0)\times [0,1]\longrightarrow B_{2\sigma}(z) $,
such that

\begin{itemize}
\item $u_i(V^\perp\cap\, \mathbf{B}_1(0),0)=\Sigma$;

\item $u_i(V^\perp\cap\, \mathbf{B}_1(0),1)=S_{\lambda_i}$;

\item $u_i(v,t)\in B_{2\sigma}(z)-Z$ for all $|v|=1$ and all  $t\in [0,1]$.
\end{itemize}

To be precise, the $u_i$ can be constructed as follows. Without loss of
generality assume:

\begin{enumerate}
\item $Z\cap B_{2\sigma}(0)=\{\exp_z(g(w),w) : w\in V\cap \, \mathbf{B}%
_1(0)\}$. Here, $g(0)=0$, so we have $|g(w)|\leq C_1|w|$ for some constant $%
C_1$ and all $w\in V\cap \, \mathbf{B}_1(0)$.

\item $\Sigma=\{\exp_z(v,f(v)) : v\in V^\perp\cap \, \mathbf{B}_1(0)\}$. We
may assume this form for any transverse slice by the implicit function
theorem. Since $f(0)=0$, we may assume, after possibly shrinking the slice,
that $|f(v)|\leq C_2$ for all $v\in V^\perp\cap \, \mathbf{B}_1(0)$, where $%
C_1C_2\leq 1/2$.

\item $S_{\lambda_i}=\{\exp_z(\lambda_i v, 0) : v\in V^\perp\cap \, \mathbf{B%
}_1(0)\}$.
\end{enumerate}

Now for $v\in V^\perp\cap \, \mathbf{B}_1(0)$, $t\in [0,1]$, set 
\begin{equation*}
u_i(v,t)=\exp_z\left( ((1-t)+t\lambda_i)v, (1-t)f(v)\right) 
\end{equation*}
Then by (ii) and (iii), the image of $u_i(\cdot, 0)$ is $\Sigma$, and the
image of $u_i(\cdot, 1)$ is $S_{\lambda_i}$. Moreover, we may assume that $%
u(v,t)\in B_{2\sigma}(x)$ by changing the radius of $\mathbf{B}_1(0)$.
Finally, note that by (ii), $|(1-t)f(v)|\leq (1-t)C_2$, whereas $%
|((1-t)+t\lambda_i)v|=(1-t)+t\lambda_i$ for $|v|=1$. If $u_i(v,t)\in Z$,
then by (i) we would have $(1-t)+t\lambda_i\leq C_1C_2(1-t)\leq (1/2)(1-t) $%
, which is impossible since $\lambda_i>0$. Hence, $u_i(v,t)\not\in Z$ for $%
|v|=1$. With this understood, we obtain 
\begin{align}
0&=\frac{1}{8\pi^2}\int_{V^\perp\cap\, \mathbf{B}_1(0)\times [0,1]}d
\left\{u_i^\ast \tr (F_{A_i}\wedge F_{A_{i}})-u_i^\ast \tr %
(F_{A_{\infty}}\wedge F_{A_{\infty}}) \right\}  \notag \\
&= \frac{1}{8\pi^2}\int_{V^\perp\cap\, \mathbf{B}_1(0)} \left\{ \tr %
(F_{A_{i,\lambda_i}}\wedge F_{A_{i,\lambda_i}})- \tr (F_{A_{\infty,%
\lambda_i}}\wedge F_{A_{\infty,\lambda_i}}) \right\}  \notag \\
&\qquad - \frac{1}{8\pi^2}\int_{\Sigma} \left\{ \tr (F_{A_{i}}\wedge
F_{A_{i}})- \tr (F_{A_{\infty}}\wedge F_{A_{\infty}}) \right\}
\label{eqn:chern-limit} \\
&\qquad + \frac{1}{8\pi^2}\int_{V^\perp\cap\, \partial\mathbf{B}_1(0)\times
[0,1]}u_i^\ast \left\{ \tr (F_{A_{i}}\wedge F_{A_{i}})- \tr %
(F_{A_{\infty}}\wedge F_{A_{\infty}}) \right\}  \notag
\end{align}
Since $A_i$ and $A_\infty$ are connections on the same bundle away from $Z$,
we may write $d_{A_i}=d_{A_\infty}+a_i$, and define the Chern-Simons
transgression, 
\begin{equation}  \label{eqn:cs}
CS(A_i, A_\infty)= \tr\bigl( a_i\wedge d_{A_\infty}(a_i)+(2/3)a_i\wedge
a_i\wedge a_i+2a_i\wedge F_{A_\infty} \bigr)
\end{equation}
Then 
\begin{align*}
\frac{1}{8\pi^2}\int_{V^\perp\cap\, \partial\mathbf{B}_1(0)\times
[0,1]}&u_i^\ast \bigl\{ \tr (F_{A_{i}}\wedge F_{A_{i}})- \tr %
(F_{A_{\infty}}\wedge F_{A_{\infty}}) \bigr\} \\
& =\frac{1}{8\pi^2}\int_{V^\perp\cap\, \partial\mathbf{B}_1(0)\times [0,1]}
u_i^\ast \left( dCS(A_i, A_\infty)\right) \\
&= \frac{1}{8\pi^2}\int_{V^\perp\cap\, \partial\mathbf{B}_1(0)}CS(A_{i,%
\lambda_i}, A_{\infty, \lambda_i}) -\frac{1}{8\pi^2}\int_{\partial%
\Sigma}CS(A_i, A_{\infty})
\end{align*}
Since $A_i\to A_\infty$ uniformly away from $Z$ and $A_{i,\lambda_i},
A_{\infty, \lambda_i}\to 0$ on compact subsets of $V^\perp-\{0\}$ (cf.\ \cite%
[p.\ 469]{HongTian04}), this term vanishes as $i\to \infty$. We conclude
from \eqref{eqn:chern-limit} that 
\begin{align*}
\lim_{i\to\infty} \frac{1}{8\pi^2}\int_{V^\perp\cap\, \mathbf{B}_1(0)} %
\bigl\{ \tr (F_{A_{i,\lambda_i}}\wedge F_{A_{i,\lambda_i}})&- \tr %
(F_{A_{\infty,\lambda_i}}\wedge F_{A_{\infty,\lambda_i}}) \bigr\} \\
&= \lim_{i\to\infty} \frac{1}{8\pi^2}\int_{\Sigma} \left\{ \tr %
(F_{A_i}\wedge F_{A_i})- \tr (F_{A_{\infty}}\wedge F_{A_{\infty}}) \right\}
\end{align*}
and the proof is complete.
\end{proof}

\subsection{Proof of Theorem \protect\ref{thm:main}}

%
%
%
%
%
%
%

Throughout this section we let $A_i=A_{t_i}$, $t_i\to\infty$ be a sequence
along the Yang-Mills flow, and suppose $A_i\to A_\infty$ is an Uhlenbeck
limit with analytic singular set $Z^{an}$ (see Theorem \ref{thm:uhlenbeck}).  Let $Z^{alg}$  denote the
algebraic singular set of the HNS filtration of the initial holomorphic
bundle $\mathcal{E}$. Then we have the following

\begin{proposition}
\label{prop:multiplicities} Let $Z$ be an irreducible codimension $2$
component of $Z^{alg}$. Then $Z\subset Z_b^{an}$ and $m_Z^{alg}=m_Z^{an}$.
Moreover, $Z_b^{an}\subset Z^{alg}$.
\end{proposition}

To isolate contributions from individual components, we will first need an
argument similar to the one used in \cite[Lemma 6]{DaskalWentworth07}.

\begin{lemma}
\label{lem:geometry} Let $Z\subset Z^{alg}$ be an irreducible codimension $2$
component. Then there exists a modification $\pi : \widehat X\to X$ with
center $\mathbf{C}$ and exceptional set $\mathbf{E}=\pi^{-1}(\mathbf{C})$,
and a filtration of $\pi^\ast\mathcal{E}$ with associated graded sheaf $\Gr%
(\pi^\ast\mathcal{E})\to\widehat X$ and singular set $\sing \Gr(\pi^\ast%
\mathcal{E})$, all with the following properties:

\begin{enumerate}
\item $\Gr(\pi^\ast\mathcal{E})\simeq\Gr(\mathcal{E})$ on $\widehat X- 
\mathbf{E}=X-\mathbf{C}$.

\item $\codim(Z\cap\mathbf{C})\geq 3$.

\item $\codim\left(\pi(\sing \Gr(\pi^\ast\mathcal{E}))-Z\right)\geq 3$.
\end{enumerate}
\end{lemma}

\begin{proof}
By Hironaka's theorem, we may find a resolution $\widehat{X}_{1}\rightarrow X
$ of the singularities of $Z^{alg}$. Note that the center of this
modification has codimension $\geq 3$ in $X$. Let $W\subset Z^{alg}$ be a
codimension $2$ irreducible component other than $Z$, and let $\widehat{W}%
_{1}$ denote the strict transform of $W$ in $\widehat{X}_{1}$. By
assumption, $\widehat{W}_{1}$ is smooth. We are going to define a sequence
of monoidal transformations 
\begin{equation*}
\widehat{X}_{n}\longrightarrow \widehat{X}_{n-1}\longrightarrow \cdots
\longrightarrow \widehat{X}_{1}\longrightarrow X
\end{equation*}%
First, let $\pi _{2}:\widehat{X}_{2}\rightarrow X$ be the blow-up of $%
\widehat{X}_{1}$ along $\widehat{W}_{1}$, and consider the induced
filtration of $\pi _{2}^{\ast }\mathcal{E}$ by saturated subsheaves with
associated graded $\Gr(\pi _{2}^{\ast }\mathcal{E})$. After a possible
further desingularization in codimension $3$, we may assume without loss of
generality that $\sing\Gr(\pi _{2}^{\ast }\mathcal{E})$ is smooth in $%
\widehat{X}_{2}$. Moreover, any codimension $2$ component of $\sing\Gr(\pi
_{2}^{\ast }\mathcal{E})$ that contains the generic $\mathbb{P}^{1}$-fiber
of the exceptional divisor of $\widehat{X}_{2}\rightarrow \widehat{X}_{1}$
projects in $X$ to a proper subvariety of $W$; hence, up to a codimension $3$
set in $X$, we may ignore these components. Let $\widehat{W}_{2}$ denote the
union of the (other) codimension $2$ components of $\sing\Gr(\pi _{2}^{\ast }%
\mathcal{E})$ in the exceptional set of $\widehat{X}_{2}\rightarrow \widehat{%
X}_{1}$. Again, $\widehat{W}_{2}$ is smooth by assumption. Define $\widehat{X%
}_{3}$ to be the blow-up of $\widehat{X}_{2}$ along $\widehat{W}_{2}$.
Repeat this process in the same manner to obtain recursively $\widehat{X}_{k}
$, for $k$ greater than $3$. We now claim that after a finite number of
steps $n$, this process stabilizes: $\widehat{W}_{n}$ is empty and $\widehat{%
X}_{n+1}=\widehat{X}_{n}$. In other words, the part of $\sing\Gr(\pi
_{n}^{\ast }\mathcal{E})$ in the exceptional set of $\widehat{X}%
_{n}\rightarrow \widehat{X}_{1}$ projects in $X$ to a proper subvariety of $W
$. Note that this then implies $\codim\left( \pi _{n}(\sing\Gr(\pi
_{n}^{\ast }\mathcal{E}))-Z\right) \geq 3$. To prove the claim it clearly
suffices to consider the case of a single step filtration: 
\begin{equation*}
0\longrightarrow \mathcal{S}\longrightarrow \mathcal{E}\longrightarrow 
\mathcal{Q}^{\ast \ast }\longrightarrow \mathcal{T}\longrightarrow 0
\end{equation*}%
where $\mathcal{S}$ and $\mathcal{E}$ are locally free at generic points of $%
W$. Let $\mathcal{I}$ be the sheaf of ideals generated by the determinants
of $\rank\mathcal{S}\times \rank\mathcal{S}$ minors of the map $\mathcal{S}%
\rightarrow \mathcal{E}$ with respect to local trivializations of both
bundles near a point $p\in W$. The vanishing set of $\mathcal{I}$ is $W$ by
definition. If $\Sigma $ is a transverse slice to $W$ at the point $p$, let $%
\mathcal{I}_{\Sigma }$ denote the ideal sheaf in $\mathcal{O}_{\Sigma }$
generated by the restriction of the generators of $\mathcal{I}$ to $\Sigma $%
. Note that  $\Ocal_\Sigma/\mathcal{I}_{\Sigma }$ is supported precisely at $p$. By 
\cite[Thm.\ 14.14, Thm. 14.13, and Thm. 14.10]{Matsumura86}, and the fact
that the stalk $\mathcal{O}_{\Sigma ,p}$ is a Cohen-Macauley local ring,
there are germs $f_{1},f_{2}\in \mathcal{I}_{\Sigma ,p}$ so that $\dim _{%
\mathbb{C}
}(\mathcal{O}_{\Sigma ,p}/<f_{1},f_{2}>)=e(\mathcal{I}_{\Sigma ,p})$, the
Hilbert-Samuel multiplicity of the ideal $\mathcal{I}_{\Sigma ,p}$, and  this latter number is constant for slices through generic smooth points
of $W$. Let $D_{1}$, $D_{2}$ be the divisors associated to $f_{1}$ and $f_{2}
$, and let $\widehat{D}_{1}$, $\widehat{D}_{2}$ be the strict transforms of $%
D_{1}$, $D_{2}$ in the blow-up of $\Sigma $ at $p$. Then the intersection
multiplicity $\langle \widehat{D}_{1},\widehat{D}_{2}\rangle $ is \emph{%
strictly} less than $\langle D_{1},D_{2}\rangle $ (cf.\ \cite[p.\ 210,
Corollary 3]{Shafarevich77}), the difference depending on the order of
vanishing of $f_{1}$ and $f_{2}$ at $p$. This means that after a fixed
number of blow-ups, depending only on $\langle D_{1},D_{2}\rangle $, $%
\widehat{D}_{1}$ and $\widehat{D}_{2}$ are disjoint. But since $D_{1}$ and $%
D_{2}$ intersect only at~$p$, the intersection multiplicity $\langle
D_{1},D_{2}\rangle $ is equal to $\dim _{%
\mathbb{C}
}(\mathcal{O}_{\Sigma ,p}/<f_{1},f_{2}>)$ by definition. It follows that
after a finite number of blow-ups $\pi _{n}:\widehat{X}_{n}\rightarrow X$ as
described above, the number depending only upon the Hilbert-Samuel
multiplicity $e(\mathcal{I}_{\Sigma ,p})$ of  a generic slice, the
strict transforms of the divisors corresponding to (the extensions of) $f_{1}
$ and $f_{2}$ in $\mathcal{I}$ intersect at most in a set $\widehat{Z}_{W}$
that projects to a proper subvariety of $W$. If $\widehat{\mathcal{S}}$ is
the saturation of $\pi _{n}^{\ast }\mathcal{S}$ in $\pi _{n}^{\ast }\mathcal{%
E}$, then $\Lambda^{\rank\mathcal{S}}\widehat{\mathcal{S}}$ is the
saturation of $\pi _{n}^{\ast }\Lambda^{\rank\mathcal{S}}\mathcal{S}$, and
so $\widehat{\mathcal{S}}$ is a subbundle away from $\widehat{Z}_{W}$. This
proves the claim. The lemma now follows by carrying out the procedure above
on all codimension $2$ components of $Z^{alg}$ other than $Z$.
\end{proof}

\begin{proof}[Proof of Proposition \protect\ref{prop:multiplicities}]
Choose an irreducible codimension $2$ component $Z\subset Z^{alg}$. 
We wish to show that $m_{Z}^{alg}=m_{Z}^{an}$.  Since $m_{Z}^{alg}\neq 0$, it will follow that $Z\subset Z^{an}$.
In fact, since $\sing A_\infty$  in the decomposition \eqref{eqn:ansingset} has codimension at least $3$, 
 $Z\subset Z_{b}^{an}$. We therefore proceed to prove the equality of multiplicities.

First, since $X$ is K\"{a}hler the rational homology class of $[Z]$ is
nonzero. Therefore, there is a class in $H_{4}(X,%
\mathbb{Q}
)$ whose intersection product with $[Z]$ is non-trivial. Since an integral
multiple of any class (not in top dimension) can be represented by an
embedded submanifold (see \cite{Thom}), we may in particular choose a
closed, oriented $4$-real dimensional submanifold $\Sigma \subset X$
representing a class $[\Sigma ]$ in $H_{4}(X,%
\mathbb{Q}
)$, so that the intersection product $[\Sigma ]\cdot \lbrack Z]\neq 0$.
Furthermore, since $\dim Z^{alg}+\dim \Sigma =\dim X$ we may choose $\Sigma $
so that it meets $Z^{alg}$ only in the smooth points of the codimension $2$
components, and this transversely. Since the intersection
multiplicity $[\Sigma]\cdot [Z]\neq 0$, we have $\Sigma \cap
Z=\{z_{1},\ldots ,z_{p}\}$ for some finite (non-empty) set of points.
Clearly, we can assume $\Sigma $ is a (positive or negatively oriented)
transverse slice at each point of intersection with $Z^{alg}$. Let $\widehat{%
X}$ and $\Gr(\pi ^{\ast }\mathcal{E})$ be as in Lemma \ref{lem:geometry}. By
transversality and part (iii) of the lemma, we may arrange so that the
strict transform $\widehat{\Sigma }$ of $\Sigma $ is embedded and $\widehat{%
\Sigma }$ intersects $\sing\Gr(\pi ^{\ast }\mathcal{E})$ only along $\pi
^{-1}(Z)$. Choose $\sigma >0$ so that for each $k=1,\ldots ,p$,

\begin{itemize}
\item $B_{2\sigma}(z_k)\subset X-\mathbf{C}$ (by Lemma \ref{lem:geometry}
(ii))

\item $B_{2\sigma}(z_k)\cap Z^{alg}\subset Z- \sing \Gr(\mathcal{E}%
)^{\ast\ast}$
\end{itemize}

Let $A_{\infty }$ be a smooth connection on $\Gr(\mathcal{E})^{\ast \ast }$
over $\cup _{k=1}^{p}B_{2\sigma }(z_{k})$, and fix a K\"{a}hler metric on $%
\widehat{X}$. By the construction in \cite{BandoSiu94} (see Theorem \ref%
{thm:bando-siu} (i)), and noting Lemma \ref{lem:geometry} (i), we can extend 
$A_{\infty }$ to a $\hat{\omega}$-admissible connection $\widehat{A}_{\infty
}$ on $\Gr(\pi ^{\ast }\mathcal{E})^{\ast \ast }$. Given a smooth connection 
$A$ on $\mathcal{E}$, let $\pi ^{\ast }A$ denote the pull-back connection on 
$\pi ^{\ast }\mathcal{E}$. Then by Theorem \ref{thm:bottchern} we have 
\begin{equation}
\frac{1}{8\pi ^{2}}\tr(F_{\pi ^{\ast }A}\wedge F_{\pi ^{\ast }A})-\frac{1}{%
8\pi ^{2}}\tr\bigl(F_{\widehat{A}_{\infty }}\wedge F_{\widehat{A}_{\infty }}%
\bigr)=\sum \hat{m}_{j}^{alg}\widehat{W}_{j}+dd^{c}\widehat{\Psi }
\label{eqn:chernhat}
\end{equation}%
where the $\widehat{W}_{j}$ are the codimension $2$ components of $\supp%
\left( \Gr(\pi ^{\ast }\mathcal{E})^{\ast \ast }/\Gr(\pi ^{\ast }\mathcal{E}%
)\right) $. Notice that by Lemma \ref{lem:geometry} (i) and the choice of $%
\Sigma $, $\Gr(\pi ^{\ast }\mathcal{E})$ is locally free in a neighborhood
of $\widehat{\Sigma }\cap \left( \cup _{k=1}^{p}B_{2\sigma }(z_{k})\right)
^{c}$, and there is one component, $\widehat{W}_{1}$ say, such that $\pi (%
\widehat{W}_{1})=Z$, while all other components $\widehat{W}_{j}$ miss $%
\widehat{\Sigma }$. Moreover, $\hat{m}_{1}^{alg}=m_{Z}^{alg}$. Now the
difference of the Chern forms for $\widehat{A}_{\infty }$ and $\pi ^{\ast }A$
are related by a Chern-Simons class $dCS(\pi ^{\ast }A,\widehat{A}_{\infty
})=dd^{c}\widehat{\Psi }$ in a neighborhood of $\widehat{\Sigma }$ away from 
$\pi ^{-1}(Z)$ (cf.\ \eqref{eqn:cs}). We can then use this fact to obtain,
by \eqref{eqn:transverse} and \eqref{eqn:chernhat}, 
\begin{align}
([\Sigma]&\cdot [Z])m_{Z}^{alg}=\frac{1}{8\pi ^{2}}\int_{\Sigma \cap \left(
\cup _{k=1}^{p}B_{\sigma }(z_{k})\right) }\left\{ \tr(F_{A}\wedge F_{A})-\tr%
(F_{A_{\infty }}\wedge F_{A_{\infty }})\right\} -\int_{\Sigma \cap \partial
\left( \cup _{k=1}^{p}B_{\sigma }(z_{k})\right) }d^{c}\widehat{\Psi }  \notag
\\
& =\frac{1}{8\pi ^{2}}\int_{\Sigma \cap \left( \cup _{k=1}^{p}B_{\sigma
}(z_{k})\right) }\left\{ \tr(F_{A}\wedge F_{A})-\tr(F_{A_{\infty }}\wedge
F_{A_{\infty }})\right\} +\int_{\widehat{\Sigma }\cap \left( \cup
_{k=1}^{p}B_{\sigma }(z_{k})\right) ^{c}}dd^{c}\widehat{\Psi }  \notag \\
& =\frac{1}{8\pi ^{2}}\int_{\Sigma \cap \left( \cup _{k=1}^{p}B_{\sigma
}(z_{k})\right) }\left\{ \tr(F_{A}\wedge F_{A})-\tr(F_{A_{\infty }}\wedge
F_{A_{\infty }})\right\} +\int_{\Sigma \cap \left( \cup _{k=1}^{p}B_{\sigma
}(z_{k})\right) ^{c}}dCS(\pi ^{\ast }A,\widehat{A}_{\infty })  \notag \\
& =\frac{1}{8\pi ^{2}}\int_{\Sigma \cap \left( \cup _{k=1}^{p}B_{\sigma
}(z_{k})\right) }\left\{ \tr(F_{A}\wedge F_{A})-\tr(F_{A_{\infty }}\wedge
F_{A_{\infty }})\right\} -\int_{\Sigma \cap \partial \left( \cup
_{k=1}^{p}B_{\sigma }(z_{k})\right) }CS(A,A_{\infty })  \label{eqn:integral}
\end{align}%
Finally, apply the above to a sequence $A_{i}$ of connections converging to $%
A_{\infty }$ away from $Z^{an}$ as in Theorem \ref{thm:tian}. Then $%
CS(A_{i},A_{\infty })\rightarrow 0$ on $\Sigma \cap \partial B_{\sigma
}(z_{k})$ for each $k$; hence the second term in \eqref{eqn:integral}
vanishes in the limit. By Lemma \ref{lem:slice}, 
\begin{equation*}
\lim_{i\rightarrow \infty }\frac{1}{8\pi ^{2}}\int_{\Sigma \cap \left( \cup
_{k=1}^{p}B_{\sigma }(z_{k})\right) }\left\{ \tr(F_{A_{i}}\wedge F_{A_{i}})-%
\tr(F_{A_{\infty }}\wedge F_{A_{\infty }})\right\} =([\Sigma]\cdot [Z]) m_{Z}^{an}
\end{equation*}%
By \eqref{eqn:integral} we therefore obtain $([\Sigma]\cdot [Z]) m_{Z}^{alg}=([\Sigma]\cdot [Z]) m_{Z}^{an}$, and since $[\Sigma]\cdot [Z]\neq
0$ we conclude that $m_{Z}^{alg}=m_{Z}^{an}$. This is the first assertion in
the statement of Proposition \ref{prop:multiplicities}. It implies that the
cycle 
\begin{equation*}
\sum_{j}m_{j}^{an}Z_{j}^{an}-\sum_{k}m_{k}^{alg}Z_{k}^{alg}
\end{equation*}%
has nonnegative coefficients. But by Corollary \ref{cor:cohomologous}, the
corresponding current is also cohomologous to zero. Hence, all codimension $2
$ components of $Z_{b}^{an}$ must in fact be contained in $Z^{alg}$ with the
same multiplicities. As mentioned previously, by the theorem of Tian and
Harvey-Shiffman, $Z_{b}^{an}$ has pure codimension $2$, and so this proves
the second statement.
\end{proof}

Proposition \ref{prop:multiplicities} proves part (3) of Theorem \ref%
{thm:main}. For part (1),  note that by
Proposition \ref{prop:multiplicities}, the irreducible codimension $2$ components 
of $Z^{alg}$ and $Z^{an}$ coincide.
By Proposition \ref{prop:scheja}, the decomposition \eqref{eqn:ansingset}, and Theorem \ref{thm:tian}, it follows that
$
Z_b^{an}-\sing A_{\infty }=Z^{alg}-\sing\Gr(\Ecal)^{\ast\ast}
$.
 We therefore need only show   that $\sing A_{\infty }=\sing\Gr(\mathcal{E})^{\ast \ast }$. Indeed, if $p\not\in \sing\Gr(%
\mathcal{E})^{\ast \ast }$, then by definition $\Gr(\mathcal{E})^{\ast \ast }
$ is locally free in a neighborhood of $p$. It follows from \cite[Theorem 2
(c)]{BandoSiu94} that the direct sum of the admissible Hermitian-Einstein
metrics on the stable summands of $\Gr(\mathcal{E})^{\ast \ast }$ is smooth
at $p$, and hence, $p\not\in \sing A_{\infty }$. Conversely, if $p\not\in %
\sing A_{\infty }$ then the direct sum of the admissible Hermitian-Einstein
metrics extends to a smooth bundle over $p$. But since $\Gr(\mathcal{E}%
)^{\ast \ast }$ is reflexive and hence normal, this implies $p\not\in \sing%
\Gr(\mathcal{E})^{\ast \ast }$. 
Given these equalities, part (1) now follows from
the decompositions \eqref{eqn:ansingset} and \eqref{eqn:algsingset} and
 Proposition \ref{prop:scheja}.
 By Remark \ref%
{rem:intro} (ii), part (2) of Theorem \ref{thm:main} also follows. The proof
of Theorem \ref{thm:main} is complete.

\noindent \frenchspacing

\bigskip 

\par\noindent {\bf Acknowledgement.}
The first author would like to thank Julien Grivaux for answering his
questions about Levy's version of the Grothendieck-Riemann-Roch theorem.
Both authors are grateful to Bo Berndtsson for his helpful comments on an argument
in a preliminary version of this paper. The first
author also acknowledges the support of the Max Planck Institut f\"{u}r
Mathematik, where part of the present manuscript was written, and where he
was staying when it took its present form.

\bigskip\bigskip\bigskip
\centerline{ERRATUM}
\bigskip

Definition \ref{def:algmult} is incorrect. Instead, the definition of the algebraic multiplicity should be the following.  Let $\Fcal$ be a coherent sheaf and $\Acal{nn}(\Fcal)\subset\Ocal_X$ the annihilator subsheaf associated to $\Fcal$. Let $\Ical$ be the radical of $\Acal{nn}(\Fcal)$. The (reduced) support of $\Fcal$ is the variety $Z$ associated to $\Ical$ with structure sheaf $\Ocal_Z=\Ocal_X/\Ical$.
By the R\"uckert Nullstellensatz, there is a positive integer $N$ such that $\Ical^N\subset\Acal{nn}(\Fcal)$.  Let 
$$
\gr \Fcal=\bigoplus_{k=0}^N \Ical^k\Fcal/\Ical^{k+1}\Fcal\ .
$$
Then for an irreducible component $Z_j\subset Z$,  define $m_j^{alg}$ to be the rank of $\gr\Fcal$ as an $\Ocal_Z$ module along $Z_j$. Notice that if $\imath: Z\hookrightarrow X$ denotes the inclusion, then locally around a smooth point of $Z$, we have $\gr\Fcal\bigr|_{Z}=\imath_\ast \imath^\ast \gr \Fcal$.
The key fact used about the  $m_j^{alg}$ is that they satisfy the formula
given in Proposition \ref{prop:key}. This in turn follows from Proposition 
\ref{prop:components}.  The statement there  is
incorrect with the definition of $m_j^{alg}$ in the paper, but it is correct with the modified definition given above, as we now explain.

Since $\Fcal$ and $\gr \Fcal$ are equal in the Grothendieck group, 
we have $\tau (\mathcal{F})=\tau (\gr\Fcal)$. This change enters in the proof of Proposition \ref{prop:components} in the following way. In case $Z$ is irreducible, the first part of the proof 
of Proposition \ref{prop:components} applies to show that
$$
\tau _{2\dim Z}(\mathcal{F}) =\tau _{2\dim Z}(\gr\Fcal)=\tau _{2\dim Z}(\imath _{\ast }\imath ^{\ast }(\gr\Fcal))
=\imath _{\ast }\tau _{2\dim Z}(\gr\Fcal\bigr|_{Z})=
\rank(\gr\Fcal\bigr|_{Z})[Z]\  .
$$
More generally,
consider the irreducible components $Z_{j}\subset 
Z$ of top dimension and their inclusions $\imath
_{j}:Z_{j}\rightarrow X$.
By the observation above, the natural map
$$
\gr\Fcal\lra  \bigoplus_j (\imath_j)_\ast\bigl(\gr\Fcal\bigr|_{Z_j}\bigr)
$$
is an isomorphism of $\Ocal_Z$-modules off a subvariety of positive codimension in $Z$.
It then follows that
\begin{align*}
\tau _{2\dim Z}(\mathcal{F)} &=\tau _{2\dim Z}(\gr\Fcal)=\tau _{2\dim Z}\bigl(\bigoplus_{j}(\imath_j)_\ast(\gr\Fcal\bigr|_{Z_j})\bigr)
=\sum_{j}(\imath _{j\ast })\bigl(\tau _{2\dim Z}(\gr\Fcal\bigr|_{Z_j})\bigr)\\
&=\sum_{j}\rank(\gr\Fcal\bigr|_{Z_{j}})[Z_{j}]=\sum_{j}m_{j}^{alg}[Z_{j}]\ ,
\end{align*}%
as claimed in Proposition \ref{prop:components}. Since this is the only
place in the paper where the definition of $m_{j}^{alg}$ is
directly used (all later instances where the $m_{j}^{alg}$ appear
refer only to the fact that they satisfy the equation of Proposition  \ref%
{prop:key}), all subsequent results in the paper are unchanged once this
modification to the original definition is made.

\bigskip 

The authors would like to thank Matei Toma for pointing out this error. 
\vfill

\end{document}